\documentclass[article,12pt,letterpaper]{amsart}

\usepackage{amsmath,amssymb,amsthm,graphicx,mathrsfs,url}
\usepackage[usenames,dvipsnames]{color}
\usepackage{hyperref}
\usepackage{amsxtra}
\usepackage{graphicx}
\usepackage{tikz}
\relax

\numberwithin{equation}{section}

\def\Im{\textrm{Im}}
\def\Re{\textrm{Re}}
\def\ov{\overline}
\def\11{{\rm 1~\hspace{-1.4ex}l} }
\def\R{\mathbb R}
\def\C{\mathbb C}
\def\Z{\mathbb Z}
\def\N{\mathbb N}

\def\T{\mathbb T}

\newcommand{\Id}{{\rm Id}}

\newcommand{\eps}{{\epsilon}}

\theoremstyle{plain}

\newtheorem{thm}{Theorem}
\newtheorem{prop}{Proposition}[section]
\newtheorem{cor}[prop]{Corollary}
\newtheorem{lemma}[prop]{Lemma}
\newtheorem{definition}[prop]{Definition}

\newtheorem{remark}[prop]{Remark}

\theoremstyle{definition}

\newtheorem{rem}[prop]{Remark}
 
\numberwithin{equation}{section}

\definecolor{darkred}{rgb}{0.7,0.1,0.1}
\definecolor{darkblue}{rgb}{0,0,0.7}
\title[Observability of Baouendi-Grushin-type equations]
{Observability of Baouendi-Grushin-type equations through resolvent estimates}
 \author[C. Letrouit]{Cyril Letrouit}
\address[C. Letrouit]{Sorbonne Universit\'e, Universit\'e Paris-Diderot SPC, CNRS, Inria, Laboratoire Jacques-Louis Lions, \'equipe CAGE, F-75005 Paris. \newline DMA, \'Ecole normale sup\'erieure, CNRS, PSL Research University, 75005 Paris.}
\email{letrouit@ljll.math.upmc.fr}
\author[C-M. Sun]{Chenmin Sun}
\address[C-M. Sun]{Universit\'e de Cergy-Pontoise, Laboratoire de Math\'ematiques AGM, UMR  8088 du CNRS, 2 av. Adolphe Chauvin
	95302 Cergy-Pontoise Cedex, France.}
\email{chenmin.sun@u-cergy.fr}

\begin{document}    

\maketitle

\begin{abstract}
In this article, we study the observability (or, equivalently, the controllability) of some subelliptic evolution equations depending on their step. This sheds light on the speed of propagation of these equations, notably in the ``degenerated directions'' of the subelliptic structure. \\
First, for any $\gamma\geq 1$, we establish a resolvent estimate for the Baouendi-Grushin-type operator $\Delta_\gamma=\partial_x^2+|x|^{2\gamma}\partial_y^2$, which has step $\gamma+1$. We then derive consequences for the observability of the Schr\"odinger type  equation $i\partial_tu-(-\Delta_\gamma)^{s}u=0$ where $s\in\N$. We identify three different cases: depending on the value of the ratio $(\gamma+1)/s$, observability may hold in arbitrarily small time, or only for sufficiently large times, or even fail for any time. \\
As a corollary of our resolvent estimate, we also obtain observability for heat-type equations $\partial_tu+(-\Delta_\gamma)^su=0$ and establish a decay rate for the damped wave equation associated with $\Delta_{\gamma}$. 
\bigskip


\bigskip

\end{abstract}


\section{Introduction and main results}

\subsection{Motivation} 
This paper addresses some issues related to the controllability and observability properties of evolution equations built on subelliptic Laplacians (or sub-Laplacians). 

Given a manifold $M$, a small subset $\omega\subset M$, a time $T>0$ and an operator $P$ (which depends on $t\in\R$ and $x\in M$), the study of controllability consists in determining whether, for any initial state $u_0$ and any final state $u_1$, there exists $f$ such that the solution of the equation
\begin{equation} \label{e:control}
Pu=\mathbf{1}_\omega f, \qquad u_{|t=0}=u_0
\end{equation}
in $M$ is equal to $u_1$ at time $T$. Of course, the functional spaces in which $u_0,u_1,f$ and the solution $u$ live have to be specified. By duality (the Hilbert Uniqueness Method, see \cite{Li88}), this controllability property is generally equivalent to some inequality of the form
\begin{equation*}
\|u(T)\|^2_X\leq C_{T,\omega}\int_0^T \|\mathbf{1}_\omega u(t)\|^2_{X}dt, \qquad \forall u_0\in X
\end{equation*}
where $u$ is the solution of $Pu=0$ with initial datum $u_0\in X$. This is called an observability inequality. In other words, controllability holds if and only if any solution of \eqref{e:control} with $f=0$ can be detected from $\omega$, in a ``quantitative way'' which is measured by the constant $C_{T,\omega}$.

Observability inequalities can be established under various assumptions and using different techniques depending on the operator $P$. In this paper, $P$ will be either a wave, a heat or a Schr\"odinger operator associated to some subelliptic Laplacian $\Delta_\gamma$, and, to establish observability inequalities, we will mainly use resolvent estimates.  They consist in a quantitative measurement of how much approximate solutions (also named quasimodes) of $\Delta_\gamma$ can concentrate away from $\omega$, and in particular resolvent estimates do not involve the time variable, at least in this context. See for example \cite{BZ04} and \cite{Mi12} for detailed studies about the link between observability and resolvent estimates.

Since the study of the controllability/observability properties of evolution equations driven by sub-Laplacians in full generality seems out of reach, in this paper we focus on a particular family of models, which we now describe.

Let $M=(-1,1)_x\times\T$, where $\T$ is the 1D torus in the $y$-variable and let $\gamma\geq 0$. We consider the Baouendi-Grushin-type sub-Laplacian $\Delta_\gamma=\partial_x^2+|x|^{2\gamma}\partial_y^2$, together with the domain
\begin{equation*}
D(\Delta_\gamma)=\{u\in \mathcal{D}'(M) : \partial_x^2 u, |x|^{2\gamma}\partial_y^2u\in L^2(M) \text{ and } u_{|\partial M}=0\}.
\end{equation*}
By H\"ormander's theorem, in the case where $\gamma\in\N$, $\Delta_\gamma$ is subelliptic, since $\partial_y$ can be obtained by taking $\gamma$ times the bracket of $\partial_x$ with $x^\gamma\partial_y$.

The observation region $\omega$ that we consider is assumed to contain a horizontal strip $(-1,1)_x\times I_y$ where $I\subset \mathbb{T}$ is a non-empty open interval of the 1D-torus. This choice for $\omega$ is natural if one is interested in understanding the specific features of propagation in the subelliptic directions (here, the vertical $y$-axis), see Section \ref{s:speed} below; this choice for $\omega$ has already been made in different but related subelliptic frameworks, see for example \cite{K17}, \cite{BuSun}, \cite{FKL}. We note, and we will come back to this point later in our analysis, that $\omega$ does not satisfy the Geometric Control Condition, which is known to be equivalent to observability of elliptic waves (see \cite{BLR}) and to imply the observability of the elliptic Schr\"odinger equation in any time (see \cite{Le92}). Several other choices for $\omega$ could have been made (see \cite{BCG} for example).

\subsection{Main results.} \label{s:mainresultssection} Our first main result is a resolvent estimate in the case $\gamma\geq 1$, which reads as follows:
\begin{thm} \label{t:resgamma}
Let $\gamma\in\R$, $\gamma\geq 1$ and let $\omega$ contain a horizontal strip $(-1,1)\times I$.  There exist $C, h_0>0$ such that for any $v\in D(\Delta_\gamma)$ and any $0<h\leq h_0$, there holds 
\begin{equation} \label{e:resgamma}
 \|v\|_{L^2(M)}\leq C(\|v\|_{L^2(\omega)}+h^{-(\gamma+1)}\|(h^2\Delta_\gamma+1)v\|_{L^2(M)}).
\end{equation}
\end{thm}
\begin{remark} 
In \cite{LL20} (see Corollary 1.9), a resolvent estimate with an exponential cost (replacing the above polynomial cost $h^{-(\gamma+1)}$) was proved for any sub-Riemannian manifold of step $k$ and for any of its subsets $\omega$ of positive Lebesgue measure. It was shown to be sharp for the Baouendi-Grushin-type sub-Laplacian $\Delta_\gamma$ (with $\gamma+1=k$) and for any open set $\omega$ whose closure does not touch the line $\{x=0\}$. Our resolvent estimate is much stronger, but heavily relies on the particular geometric situation under study.
\end{remark}
\begin{remark}
From the proof of (2) of Theorem \ref{t:main}, the resolvent estimate \eqref{e:resgamma} is sharp 
in the sense that there exists a sequence of quasi-modes $v_h$ which saturates the inequality. Indeed, a better resolvent estimate than \eqref{e:resgamma}, together with \cite[Theorem 4]{BZ04}, would contradict the lack of observability for short times in Point (2) of Theorem \ref{t:main} (see the argument after Theorem \ref{abstract} in Section \ref{s:locobssection}).

Furthermore, the conclusion of Theorem \ref{t:resgamma} does not apply to the case $\gamma<1$, at least if we remove the boundary. For example, when $\gamma=0$ and $\Delta_0$ is the usual Laplace operator on the torus $\T^2$, it follows from \cite{BLR} that the resolvent estimate \eqref{e:resgamma} with order $O(h^{-1})$ cannot hold if $\omega$ does not satisfy the geometric control condition with respect to the geodesic flow.
\end{remark}

In this paper, we will explore the consequences of this resolvent estimate for the observability of evolution equations driven by $\Delta_\gamma$.

Let us consider the  Schr\"odinger-type equation with Dirichlet boundary conditions
\begin{equation} \label{e:schrodfrac}
\left\lbrace \begin{array}{l}
i\partial_tu-(-\Delta_\gamma)^{s}u=0 \\
u_{|t=0}=u_0 \in L^2(M) \\
u_{|x=\pm 1}=0
\end{array}\right.
\end{equation}
where $s\in\N$ is a fixed integer and $\gamma\geq 0$, $\gamma\in \R$. 
Here $(-\Delta_{\gamma})^s$ is defined ``spectrally'' by its action on eigenspaces of the operator $\Delta_{\gamma}$ associated with Dirichlet boundary conditions. In other words, by classical embedding theorems (recalled in Lemma \ref{embedding}), $(\Delta_\gamma,D(\Delta_\gamma))$ has a compact resolvent, and thus there exists an orthonormal Hilbert basis of eigenfunctions $(\varphi_j)_{j\in\N}$ such that $-\Delta_\gamma\varphi_j=\lambda_j^2\varphi_j$, with the $\lambda_j$ sorted in increasing order. The domain of $(-\Delta_{\gamma})^s$ is given by
\begin{equation}  \label{e:domainpowers} 
D((-\Delta_{\gamma})^s)=\{u\in L^2(M): \sum_{j\in\N} \lambda_j^{4s}|( u,\varphi_j)_{L^2(M)}|^2<\infty   \}.
\end{equation}
Note that a function $u$ in $ D((-\Delta_\gamma)^s)$ verifies the boundary conditions
\begin{equation} \label{e:bdrycond}
(-\Delta_{\gamma})^ku|_{\partial M}=0, \qquad \text{for any } 0\leq k<s-\frac{1}{4}.
\end{equation}
 In Appendix \ref{a:wellposed}, we prove this fact and we also show that \eqref{e:schrodfrac} is well-posed in $L^2(M)$. Of course, the solution of \eqref{e:schrodfrac} does not live in general in the energy space given by the form domain of $(-\Delta_\gamma)^s$, but only in $L^2(M)$.

Given an open subset $\widetilde{\omega}\subset M$, we say that \eqref{e:schrodfrac} is observable in time $T_0>0$ in $\widetilde{\omega}$ if there exists $C>0$ such that for any $u_0\in L^2(M)$, there holds
\begin{equation} \label{e:obsschrod}
\|u_0\|_{L^2(M)}^2\leq C\int_0^{T_0} \|e^{-it(-\Delta_\gamma)^s}u_0\|_{L^2(\widetilde{\omega})}^2 dt.
\end{equation}
Our second main result roughly says that observability holds if and only if the subellipticity, measured by the step $\gamma+1$, is not too strong compared to $s$:
\begin{thm} \label{t:main}
Assume that $\gamma\in\R$, $\gamma\geq 1$. Let $I\subsetneq \T_y$ be a strict open subset, and let $\omega=(-1,1)_x\times I$. Then, for $s\in\N$, we have:
\begin{enumerate}
\item If $\frac12(\gamma+1)<s$, \eqref{e:schrodfrac} is observable in $\omega$ for any $T_0>0$;
\item If  $\frac12(\gamma+1)=s$, there exists $T_{\inf}>0$ such that \eqref{e:schrodfrac} is observable in $\omega$ for $T_0$ if and only if $T_0\geq T_{\inf}$;
\item If $\frac12(\gamma+1)>s$, for any $T_0>0$, \eqref{e:schrodfrac} is not observable in $\omega$.
\end{enumerate}
Indeed, Points (1) and (2) hold under the weaker assumption that $\omega$ contains a horizontal band of the form $(-1,1)_x\times I$; and Point (3) holds under the weaker assumption that $M\setminus \omega$ contains an open neighborhood of some point $(x,y)\in M$ with $x=0$.
\end{thm}

Let us make several comments about this result:
\begin{itemize}
\item In the case $\frac12(\gamma+1)=s$, our proof only provides a lower bound on $T_{\inf}$ (see Remark \ref{r:lowertime}). The exact value of $T_{\inf}$ was explicitly computed in \cite{BuSun} in the case $\gamma=s=1$. It is an interesting problem to compute this exact value for $s,\gamma$ satisfying $s=\frac{1}{2}(\gamma+1)$, and more importantly, to give a geometric interpretation for this exact constant in a more general subelliptic setting.
\item  The number $\frac12(\gamma+1)$ appearing in Theorem \ref{t:main} is already known to play a key role in many other problems. Recall that the step of the manifold (defined as the least number of brackets required to generate the whole tangent space) is equal to $\gamma+1$ (when $\gamma\in\N$). Then, $2/(\gamma+1)$ is the exponent known as the gain of Sobolev derivatives in subelliptic estimates. Note that $\frac12(\gamma+1)$ is also the threshold found in the work \cite{BCG} which deals with observability of the heat equation with sub-Laplacian $\Delta_\gamma$, and that it is related to the growth of eigenvalues for the operator $-\partial_x^2+x^{2\gamma}$, see for example Section 2.3 in \cite{BCG}.
\item In the statement of Theorem \ref{t:main}, we took $s\in\N$ in order to avoid technical issues of non-local effects due to the fractional Laplacian. We expect that the statements in Theorem \ref{t:main} are also true for all $s>0$.
\item The assumption that $\gamma\geq 1$ for Points (1) and (2) is mainly due to the technical issue that the Hamiltonian flow associated with the symbol $\partial_x^2+|x|^{2\gamma}\partial_y^2$ may not be unique if $0< \gamma<1$ (see Section \ref{s:gcc}). Dealing with this case, and more generally addressing the question of propagation of singularities for metrics with lower regularity, is an open problem.
\end{itemize}
We now derive from Theorems \ref{t:resgamma} and \ref{t:main} two consequences. First, Theorem \ref{t:main} implies the following result about observability of heat-type equations associated to $\Delta_\gamma$ (which are well-posed, as proved in Appendix \ref{a:wellposed}):
\begin{cor} Assume that $\gamma\in\R$, $\gamma\geq 1$ and $\omega$ contains a horizontal strip $(-1,1)_x\times I$. For any $s\in\N$, $s>\frac12(\gamma+1)$ and any $T_0>0$, final observability for the heat equation with Dirichlet boundary conditions
\begin{equation} \label{e:heatgamma}
\left\lbrace \begin{array}{l}
\partial_tu+(-\Delta_\gamma)^s u=0 \\
u_{|t=0}=u_0 \in L^2(M) \\
u_{|x=\pm 1}=0
\end{array}\right.
\end{equation}
holds in time $T_0$. In other words, there exists $C>0$ such that for any $u_0\in L^2(M)$, there holds
\begin{equation*}
\|e^{-T_0(-\Delta_\gamma)^s}u_0\|_{L^2(M)}^2\leq \int_0^{T_0} \|e^{-t(-\Delta_\gamma)^s}u_0\|_{L^2(\omega)}^2dt.
\end{equation*}
\end{cor}
This is a direct consequence of Corollary 2 in \cite{DM12} and Point (2) of Theorem \ref{t:main}. Note also that observability for \eqref{e:heatgamma} fails for any time if $\gamma=s=1$ (see \cite{K17}), so that we cannot expect that an analogue of Point (2) of Theorem \ref{t:main} holds for heat-type equations. This last fact - observability of a Schr\"odinger semigroup while the associated heat semigroup is not observable - gives an illustration of Proposition 3 of \cite{DM12} (which states that the same phenomenon occurs for the harmonic oscillator on the real line observed in a set of the form $(-\infty,x_0)$, $x_0\in\R$).

Finally, Theorem \ref{t:resgamma} also implies a decay rate for the damped wave equation associated to $\Delta_\gamma$. To state it, we introduce the following adapted Sobolev spaces: for $k=1,2$,
\begin{equation*}
H_{\gamma}^k(M)=\{v\in \mathcal{D}'(M), \ (-\Delta_\gamma+1)^{k/2}v\in L^2(M)\}, \qquad \|v\|_{H_{\gamma}^k(M)}=\|(-\Delta_\gamma+1)^{k/2}v\|_{L^2(M)}
\end{equation*}
and 
$H_{\gamma,0}^1(M)$ is the completion of $C_c^\infty(M)$ for the norm $\|\cdot\|_{H_\gamma^1(M)}$. 

Let $b\in L^{\infty}(M), b\geq 0$ such that $\inf_{q\in\ov{\omega}}b(q)>0$.  On the space $\mathcal{H}:=H_{\gamma,0}^1(M)\times L^2(M)$, the operator $$\mathcal{A}=\begin{pmatrix} 0 & 1 \\ \Delta_\gamma & -b\end{pmatrix}$$ with domain $D(\mathcal{A})=(H_{\gamma}^2(M)\cap H_{\gamma,0}^1(M))\times H_{\gamma,0}^1(M)$ generates a bounded semigroup (from the Hille-Yosida theorem) and the damped wave equation 
\begin{equation} \label{e:damped}
(\partial_t^2-\Delta_\gamma+b\partial_t)u=0
\end{equation}
with Dirichlet boundary conditions and given initial datum $(u_0,u_1)\in\mathcal{H}$ admits a unique solution $u\in C^0(\R^+;H_{\gamma,0}^1(M))\cap C^1(\R^+; L^2(M))$, see Appendix \ref{a:wellposed}.
\begin{cor} \label{c:damped}
Assume $\gamma\in\R$, $\gamma\geq 1$ and $\omega$ contains a horizontal strip $(-1,1)_x\times I$. There exists $C>0$ such that, for any $(u_0,u_1)\in D(\mathcal{A})$, the solution $u(t)$ of \eqref{e:damped}
with initial conditions $(u,\partial_tu)_{|t=0}=(u_0,u_1)$ satisfies
\begin{equation} \label{e:decayrate}
E(u(t),\partial_tu(t))^{\frac12}\leq \frac{C}{t^{\frac{1}{2\gamma}}} E(\mathcal{A}(u_0,u_1))^{\frac12}
\end{equation}
for any $t\geq 1$, where
\begin{equation*}
E(v,w)=\|\partial_xv\|^2_{L^2(M)}+\||x|^\gamma\partial_yv\|^2_{L^2(M)}+\|w\|^2_{L^2(M)}.
\end{equation*}
In particular, $E(u(t),\partial_tu(t))\rightarrow 0$ as $t\rightarrow +\infty$.
\end{cor}
\begin{remark}
As usual for the damped wave equation, one cannot replace $E(\mathcal{A}(u_0,u_1))^{\frac12}$ in the r.h.s. of  \eqref{e:decayrate} by $ E(u_0,u_1)^{\frac12}$, otherwise the rate $t^{-\frac{1}{2\gamma}}$ could be improved to an exponential decay.
\end{remark}
The proof of this corollary from Theorem \ref{t:resgamma} is essentially contained in Proposition 2.4 of \cite{AL14}. To be self-contained, we prove Corollary \ref{c:damped} in Appendix \ref{s:damped}. Note that the decay rate $t^{-\frac{1}{2}}$ when $\gamma=1$ is not new. This special case is a direct consequence of the Schr\"odinger observability proved in \cite{BuSun} and an abstract result (Theorem 2.3) in \cite{AL14}, linking the Schr\"odinger observability and the decay rate of the associated damped wave equation. However, when $\gamma>1$, the Schr\"odinger equation is not observable ((3) of Theorem \ref{t:main}), and we have to apply Theorem \ref{t:resgamma}. 
Also, we do not address here the question of the optimality of the decay rate given by Corollary \ref{c:damped}. See \cite[Section 2C]{AL14} for other open questions related to decay rates of damped waves.

\subsection{Comments and sketch of proof} \label{s:speed}
Let us describe in a few words the intuition underlying our results, notably Theorem \ref{t:main}. For that, we start with the case $s=1/2$ (corresponding to wave equations) which, although not covered by Theorem \ref{t:main}, is of interest. Whereas elliptic wave equations are observable in finite time under a condition of geometric control (\cite{BLR}), it is known that for (strictly) subelliptic wave equations, observability fails in any time (\cite{Let}).  This is due to the fact that in (co)-directions where the sub-Laplacian is not elliptic, the propagation of waves, and more generally of any evolution equation built with sub-Laplacians, is slowed down. On the other side, large $s$ correspond to a quicker propagation along all directions. Therefore, Theorem \ref{t:main} characterizes the threshold for the ratio of $\gamma$ and $s$ to get an exact balance between subelliptic effects (measured by the step $\gamma+1$) and elliptic phenomena (measured by $s$), and thus ``finite speed of propagation'' along subelliptic directions.

This same analysis underlies the result on the Baouendi-Grushin-Schr\"odinger equation \cite{BuSun}, which was the starting point of our analysis: indeed, \cite{BuSun} deals with the critical case $\frac12(\gamma+1)=s=1$. Although the elliptic Schr\"odinger equation propagates at infinite speed, in subelliptic geometries, observability may hold only for sufficiently large time or even fail in any time if the degeneracy measured by $\gamma$ is sufficiently strong. To our knowledge, the paper \cite{BGX00}, which exhibited a family of travelling waves solutions of the Schr\"odinger equation \eqref{e:schrodfrac} for $\gamma=1$, moving at speeds proportional to $n\in\N$, was the first result showing the slowdown of propagation in degenerate directions.

\smallskip

The paper is organized as follows. 

In Section \ref{s:proofresgamma}, we prove Theorem \ref{t:resgamma}, roughly following the same lines as in \cite{BuSun}. Due to the absence of the time-variable in our resolvent estimate, our proof is however slightly simpler, but as a counterpart, our method does not allow us to compute explicitly the minimal time $T_{\inf}$ of observability in Point (2) of Theorem \ref{t:main}. After having spectrally localized the sub-Laplacian $\Delta_\gamma$ around $h^{-2}$, our proof relies on a careful analysis of several regimes of comparison between $|D_y|$ and $\Delta_\gamma$, which roughly correspond to different types of trajectories for the geodesics in $M$: we split the function $v$ appearing in \eqref{e:resgamma} according to Fourier modes in $y$ and then we establish estimates for different ``spectrally localized'' parts of $v$ of the form $\psi(h^2\Delta_\gamma)\chi_h(D_y)v$. Here, $\chi_h$ localizes $D_y$ in some subinterval of $\R$ which depends on $h$. Fixing a small constant $b_0\ll 1$, the three different regimes which we distinguish are:
\begin{itemize}
\item the degenerate regime in Section \ref{s:degenerate} ($|D_y|\geq b_0^{-1}h^{-1}$), for which we use a positive commutator method (also known as ``energy method'', and used for example to prove propagation of singularities in the literature, see \cite[Section 3.5]{Ho71});
\item the regime of the geometric control condition in Section \ref{s:gcc} ($b_0^{-1}h^{-1}\geq |D_y|\geq b_0h^{-1}$), handled with semi-classical defect measures;
\item the regime of horizontal propagation ($|D_y|\leq b_0h^{-1}$) in Sections \ref{s:hor1} and \ref{s:hor2}, for which we use a positive commutator argument, and then a normal form method.
\end{itemize}

 In Section \ref{s:proof12}, using the link between resolvent estimates and observability of Schr\"odinger-type semigroups established in \cite{BZ04}, we deduce Points (1) and (2) of Theorem \ref{t:main} from Theorem \ref{t:resgamma}. Indeed, we first establish a \textit{spectrally localized} observability inequality, from which we deduce the full observability using a classical procedure described for example in \cite{BZ12}.
 
 In Section \ref{s:Point3}, we prove Point (3) of Theorem \ref{t:main}. For that, we construct a sequence of approximate solutions of \eqref{e:schrodfrac} whose energy concentrates on a point $(x,y)\in (-1,1)\times \T$ with $x=0$ and $y\notin I$. The existence of such a sequence contradicts the observability inequality \eqref{e:obsschrod} and is possible only when $\frac12(\gamma+1)>s$. For constructing the sequence of initial data, we add in a careful way the ground states of the operators $-\partial_x^2+|x|^{2\gamma}\eta^2$ for different $\eta$'s (the Fourier variable of $y$). These initial data propagate at nearly null speed along the vertical axis $x=0$.
 
 Finally, in Appendix \ref{a:wellposed}, we prove the well-posedness of the Schr\"odinger-type equation \eqref{e:schrodfrac}, the heat-type equation \eqref{e:heatgamma} and the damped wave equation \eqref{e:damped}, using standard techniques such as the Hille-Yosida theorem.
 In Appendix \ref{s:damped}, we prove Corollary \ref{c:damped}. Using results of \cite{BT10}, it is sufficient to estimate the size of $(i\lambda{\rm Id}-\mathcal{A})^{-1}$ for large $\lambda\in\R$ (and in appropriate functional spaces). This is done mainly thanks to a priori estimates on the system $(i\lambda{\rm Id}-\mathcal{A})U=F$, and using the resolvent estimate of Theorem \ref{t:resgamma}.

\smallskip

\textbf{Acknowledgments.} We thank Camille Laurent for interesting discussions, and an anonymous referee for numerous suggestions which improved the paper. C.~L. was partially supported by the grant ANR-15-CE40-0018 of the ANR (project SRGI). C.~S is supported by the postdoc
program: ``Initiative d'Excellence Paris Seine`` of CY Cergy-Paris Universit\'e and ANR grant
ODA (ANR-18-CE40- 0020-01).

\section{Proof of Theorem \ref{t:resgamma}} \label{s:proofresgamma}
This section is devoted to the proof of Theorem \ref{t:resgamma}. In all the sequel, $\gamma\geq 1$ is fixed. It is sufficient to deal with the case where $\omega=(-1,1)_x\times I$ where $I$ is a simple interval, since if Theorem \ref{t:resgamma} holds for some $\omega=\omega_1$, then it holds for any $\omega_2 \supset \omega_1$. Hence, in all the sequel, we assume that $I$ is a simple interval $(a_1,a_2)$. Also, we use the notations $D_x=\frac{1}{i}\partial_x$ and $D_y=\frac{1}{i}\partial_y$.

We will argue by contradiction.  Assume that there exists a sequence $(v_h)_{h>0}$ such that
\begin{equation} \label{e:seekcontrad}
\|v_h\|_{L^2(M)}=1, \quad \|v_h\|_{L^2(\omega)}=o(1), \quad \|f_h\|_{L^2(M)}=o(h^{\gamma+1})
\end{equation}
where $f_h=(h^2\Delta_\gamma+1)v_h$, and we seek for a contradiction, which would prove Theorem \ref{t:resgamma}. Let us show that we can furthermore assume that $v_h$ has localized spectrum: for that, we consider an even cutoff $\psi\in C_c^\infty(\R)$, such that $\psi\equiv 1$ near $\pm 1$ and $\psi=0$ outside $(-2,-\frac{1}{2})\cup (\frac{1}{2},2)$. We set $w_h=(1-\psi(h^2\Delta_\gamma))v_h$. Then $(h^2\Delta_\gamma+1)w_h=(1-\psi(h^2\Delta_\gamma))f_h$ has $L^2$ norm which is $o(h^{\gamma+1})$. Moreover, we also deduce that $w_h=(h^2\Delta_\gamma+1)^{-1}(1-\psi(h^2\Delta_\gamma))f_h$ and since $(h^2\Delta_\gamma+1)^{-1}(1-\psi(h^2\Delta_\gamma))$ is elliptic and thus bounded from $L^2(M)$ to $L^2(M)$, we obtain $\|w_h\|_{L^2(M)}=o(1)$. Hence, considering $v_h-w_h$ instead of $v_h$, we can furthermore assume that $v_h=\psi(h^2\Delta_\gamma)v_h$.

In the next subsections, we use a decomposition of $v_h$ as $v_h=v_h^1+v_h^2+v_h^3+v_h^4$ where
\begin{align*}
v_h^1=(1-\chi_0(b_0hD_y))v_h,&\qquad  v_h^2=(\chi_0(b_0hD_y)-\chi_0(b_0^{-1}hD_y))v_h \\
v_h^3=(\chi_0(b_0^{-1}hD_y)-\chi_0(h^{\epsilon}D_y))v_h,& \qquad v_h^4=\chi_0(h^{\epsilon}D_y)v_h,
\end{align*}
where $0<\epsilon\ll 1,0<b_0\ll 1$ are small parameters which will be fixed throughout the article and will be specified later (respectively in Proposition \ref{Observation-Lowfrequency} and in Lemma \ref{ellipticregularity}).
This is a decomposition according to the dual Fourier variable of $y$ and defined by functional calculus. The cut-off $\chi_0\in C_c^\infty(\R)$ will be defined later (see \eqref{chi0}).  We prove that $v_h^j=o(1)$ for $j=1,2,3,4$, which contradicts \eqref{e:seekcontrad}. The methods used for each $j$ are quite different, and roughly correspond to the different behaviours of geodesics according to their momentum $\eta\sim D_y$.


\subsection{A priori estimate and elliptic regularity}
 We start with the following coercivity estimate:
\begin{lemma}\label{l:dy}
There exists $C_1>0$ such that for any $u$, the following inequality holds:
\begin{equation*}
\||D_y|^{\frac{2}{\gamma+1}} u\|_{L^2(M)}\leq C_1\|\Delta_\gamma u\|_{L^2(M)}
\end{equation*}
\end{lemma}
\begin{proof}[Proof of Lemma \ref{l:dy}]
We write a Fourier expansion in $y$: for $\eta\in\Z$, we set $\widehat{u}_\eta(\cdot):=\mathcal{F}_y(u)(\cdot,\eta)$. Then, we have
\begin{align*}
\mathcal{F}_y(-\Delta_\gamma u)(x,\eta)=(D_x^2+|x|^{2\gamma}\eta^2)\widehat{u}_\eta(x).
\end{align*}
We make the change of variables $z=|\eta|^{\frac{1}{\gamma+1}}x$, and we set $f(z,\eta)=\mathcal{F}_y(-\Delta_\gamma u)(x,\eta)$ and $\widehat{v}_\eta(z)=\widehat{u}_\eta(x)$. Then we obtain
\begin{equation*}
f(z,\eta)=|\eta|^{\frac{2}{\gamma+1}}(D_z^2+|z|^{2\gamma})\widehat{v}_\eta(z),
\end{equation*}
and thus, using that $D_z^2+|z|^{2\gamma}$ is elliptic (since its spectrum is strictly above $0$), we get
\begin{equation*}
|\eta|^{\frac{2}{\gamma+1}}\|\widehat{v}_\eta\|_{L^2_z}\leq C \|f(\cdot,\eta)\|_{L^2_z}
\end{equation*}
for some constant $C>0$ (independent of $\eta$). Coming back to the $x$ variable and summing over $\eta$, we obtain
\begin{align*}
\||D_y|^{\frac{2}{\gamma+1}} u\|_{L^2(M)}^2&=\sum_{\eta\in\Z}|\eta|^{\frac{4}{\gamma+1}}\|\widehat{u}_\eta\|^2_{L^2_x} \\
&\leq C_1 \sum_{\eta\in\Z}\|\mathcal{F}_y(-\Delta_\gamma u)(\cdot,\eta)\|_{L^2_x}^2\\
&=C_1 \|\Delta_\gamma u\|_{L^2(M)}^2
\end{align*}
thanks to Plancherel formula, which finishes the proof.
\end{proof}

Let $\chi_0\in C_c^\infty(\R; [0,1])$ such that
\begin{equation}\label{chi0}
\chi_0(\zeta)\equiv 1, \text{ if } |\zeta|\leq (4C_1)^{\frac{\gamma+1}{2}} \text{ and } \chi_0(\zeta)\equiv 0 \text{ if } |\zeta|>(8C_1)^{\frac{\gamma+1}{2}}.
\end{equation}
\begin{cor} \label{c:locdy}
For $0<h<1$, there holds
\begin{equation*}
\psi(h^2\Delta_\gamma)(1-\chi_0(h^{\gamma+1}D_y))=0.
\end{equation*}
\end{cor}
\begin{proof}
For $n\in\Z$, we consider an Hilbert basis of eigenfunctions $\varphi_{m,n}$ of $L^2_x$ satisfying
\begin{equation}  \label{e:eigvarphi}
(D_x^2+|x|^{2\gamma}n^2)\varphi_{m,n}=\lambda_{m,n}^2\varphi_{m,n},\quad \|\varphi_{m,n}(x)\|_{L^2((-1,1))}=1,
\end{equation}
so that $\varphi_{m,n}e^{iny}$ is an eigenfunction of $\Delta_\gamma$ with associated eigenvalue $-\lambda_{m,n}^2$. Let $f\in D(\Delta_\gamma)$, and consider $f_h=\psi(h^2\Delta_\gamma)(1-\chi_0(h^{\gamma+1}D_y))f$. We write
\begin{equation*}
f_h=\sum_{m,n}a_{m,n}\psi(-h^2\lambda_{m,n}^2)(1-\chi_0(h^{\gamma+1}n))\varphi_{m,n}e^{iny}.
\end{equation*}
We use Plancherel formula, apply Lemma \ref{l:dy} to $f_h$ and we obtain
\begin{equation} \label{e:compsum}
\sum_{m,n}|n|^{\frac{4}{\gamma+1}}|a_{m,n}|^2\psi(-h^2\lambda_{m,n}^2)^2(1-\chi_0(h^{\gamma+1}n))^2\leq C_1^2\sum_{m,n}\lambda_{m,n}^4|a_{m,n}|^2\psi(-h^2\lambda_{m,n}^2)^2(1-\chi_0(h^{\gamma+1}n))^2
\end{equation}
On the support of $\theta(h,m,n):=\psi(-h^2\lambda_{m,n}^2)^2(1-\chi_0(h^{\gamma+1}n))^2$, there holds $|n|^{\frac{4}{\gamma+1}}\geq 16C_1^2h^{-4}>C_1^2\lambda_{m,n}^4$. Indeed, for the first inequality, we used the support properties of $\chi_0$, and for the second the support of $\psi$. This contradicts \eqref{e:compsum}, except if all $a_{m,n}$ vanish, i.e., $f_h\equiv 0$. 
\end{proof}

Corollary \ref{c:locdy} implies that 
\begin{equation} \label{e:locvcoercivity}
v_h=\psi(h^2\Delta_\gamma)\chi_0(h^{\gamma+1}D_y)v_h.
\end{equation}

The next lemma shows that in the regime $|D_y|\gg h^{-1}$, the energy of $v_h$ concentrates in the region $|x|\ll 1$. Let $\chi\in C_c^{\infty}(\R)$ such that $\chi(\zeta) \equiv 1$ for $|\zeta|\leq 2^{\frac{1}{\gamma}}$.  Also, possibly taking a larger $C_1$ in Lemma \ref{l:dy}, we can assume that $C_1\geq 1$: in particular, $\chi_0(\zeta)\equiv 1$ for $|\zeta|\leq 1$.
\begin{lemma}[Elliptic regularity]\label{ellipticregularity} 
There exist small constants $0<h_0\ll 1$ and $0<b_0\ll 1$ such that for all $0<h<h_0$, there holds
\begin{align*}
 &\big\|\big(1-\chi(b_0^{-\frac{1}{\gamma}}x)\big)(1-\chi_0(b_0hD_y))v_h\big\|_{L^2(M)}+\big\|\big(1-\chi(b_0^{-\frac{1}{\gamma}}x)\big)(1-\chi_0(b_0hD_y))h\partial_xv_h\big\|_{L^2(M)}\\ \leq & C_Nh^N\Big(\|v_h\|_{L^2(M)}+\|h\nabla_{\gamma}v_h\|_{L^2(M)}\Big),
\end{align*}	
for any $N\in\N$.
\end{lemma}

\begin{proof}
As in the previous lemma, we write the eigenfunction expansion of $v_h$ as $$(1-\chi_0(b_0hD_y))v_h=\sum_{\substack{m,n:|n|\geq b_0^{-1}h^{-1}\\ \frac{1}{\sqrt{2}}h^{-1}\leq \lambda_{m,n}\leq\sqrt{2}h^{-1}} }a_{m,n}\mathrm{e}^{iny}\varphi_{m,n}(x)$$
since $\chi_0(\zeta)\equiv 1$ for $|\zeta|\leq 1$ and $v_h=\psi(h^2\Delta_\gamma)v_h$.

We claim that it suffices to prove:
\begin{align}\label{eq:elliptic} 
 \|(1-\chi(b_0^{-\frac{1}{\gamma}}x))\varphi_{m,n}\|_{L^2}+\|(1-\chi(b_0^{-\frac{1}{\gamma}}x))h\partial_x\varphi_{m,n}\|_{L^2} \leq C_Nh^N
\end{align}
for all $N\in\N$ and $m,n$ such that $\frac{1}{\sqrt{2}}h^{-1}\leq \lambda_{m,n}\leq \sqrt{2}h^{-1}, |n|\geq b_0^{-1}h^{-1}$.
Indeed, Cauchy-Schwarz and \eqref{eq:elliptic} together imply
\begin{align*}
&\big\|\big(1-\chi(b_0^{-\frac{1}{\gamma}}x)\big)(1-\chi_0(b_0hD_y))v_h \big\|_{L^2(M)}\\ \leq &\sum_{\substack{m,n: b_0^{-1}h^{-1}\leq |n|\leq Ch^{-(\gamma+1)} \\
\frac{1}{\sqrt{2}}h^{-1}\lambda_{m,n}\leq \sqrt{2}h^{-1} } }|a_{m,n}|\big\|(1-\chi(b_0^{-\frac{1}{\gamma}}x))\varphi_{m,n} \big\|_{L^2}\\
\leq &C_Nh^N\|v_h\|_{L^2}(\#\{(m,n):  b_0^{-1}h^{-1}\leq |n|\leq Ch^{-(\gamma+1)}, \frac{1}{\sqrt{2}}h^{-1}\leq \lambda_{m,n}\leq \sqrt{2}h^{-1} \} )^{1/2}.
\end{align*}
Since $\lambda_{m,n}=|n|^{\frac{2}{\gamma+1}}\mu_{m,n}$ where $\mu_{m,n}$ is the $m$-th eigenvalue of the operator $D_z^2+|z|^{2\gamma}$ on $L^2(|z|\leq |n|^{\frac{1}{\gamma+1}})$ with Dirichlet boundary condition, we deduce from Weyl's law that
$$ \#\{(m,n):  b_0^{-1}h^{-1}\leq |n|\leq Ch^{-(\gamma+1)}, \frac{1}{\sqrt{2}}h^{-1}\leq\lambda_{m,n}\leq \sqrt{2}h^{-1} \}\leq Ch^{-N_0}
$$
for some $N_0\in\N$.\footnote{To obtain this rough estimate, it suffices to apply Weyl's law for each fixed $n$ and count the number of $n$.} Therefore, it is sufficient to establish \eqref{eq:elliptic}, which roughly says that in the regime we consider, the energy of eigenfunctions concentrates near $x=0$.

Multiplying \eqref{e:eigvarphi} by $(1-\chi(b_0^{-\frac{1}{\gamma}}x))^2\ov{\varphi}_{m,n}$ and integrating over $x\in(-1,1)$, we obtain that
\begin{align*}
\int_{-1}^{1}(1-\chi(b_0^{-\frac{1}{\gamma}}x))^2\lambda_{m,n}^2|\varphi_{m,n}(x)|^2dx=\int_{-1}^{1}(1-\chi(b_0^{-\frac{1}{\gamma}}x))^2\ov{\varphi}_{m,n}(x)\cdot (-\partial_x^2+|x|^{2\gamma}n^2)\varphi_{m,n} dx.
\end{align*} 
Doing integration by part for the r.h.s., and using the fact that $n^2|x|^{2\gamma}\geq \frac{4}{h^2} $ on the support of $1-\chi(b_0^{-\frac{1}{\gamma}}x)$ when $|n|\geq b_0^{-1}h^{-1}$, we deduce that the r.h.s. can be bounded from below by
\begin{align*}
&\frac{4}{h^2} \int_{-1}^{1}(1-\chi(b_0^{-\frac{1}{\gamma}}x))^2|\varphi_{m,n}(x)|^2dx+\int_{-1}^1 (1-\chi(b_0^{-\frac{1}{\gamma}}x))^2|\partial_x \varphi_{m,n}(x)|^2dx\\-&\int_{-1}^12b_0^{-\frac{1}{\gamma}}\chi'(b_0^{-\frac{1}{\gamma}}x)(1-\chi(b_0^{-\frac{1}{\gamma}}x) )\ov{\varphi}_{m,n}(x)\partial_x\varphi_{m,n}(x)dx.
\end{align*}
Using the fact that $\frac{4}{h^2}-\lambda_{m,n}^2\geq \frac{2}{h^2}$, we obtain that
\begin{equation}\label{induction0}
\begin{aligned}
&2h^{-2}\|(1-\chi(b_0^{-\frac{1}{\gamma}}x))\varphi_{m,n} \|_{L^2}^2+\|(1-\chi(b_0^{-\frac{1}{\gamma}}x))\partial_x\varphi_{m,n} \|_{L^2}^2 \\ \leq &Cb_0^{-\frac{1}{\gamma}}\|\chi'(b_0^{-\frac{1}{\gamma}}x)\varphi_{m,n}\|_{L^2}\|(1-\chi(b_0^{-\frac{1}{\gamma}}x) )\partial_x\varphi_{m,n}\|_{L^2}.
\end{aligned}
\end{equation}
Using Young's inequality in the r.h.s., this implies
\begin{align*}  \|(1-\chi(b_0^{-\frac{1}{\gamma}}x))\varphi_{m,n}\|_{L^2}+\|(1-\chi(b_0^{-\frac{1}{\gamma}}x))h\partial_x\varphi_{m,n}\|\leq Cb_0^{-\frac{1}{\gamma}}h.
\end{align*}
To prove a better estimate, i.e. with an $h^N$ in the r.h.s. instead of $h$, we observe that 
\begin{align*}
 \|\chi'(b_0^{-\frac{1}{\gamma}}x)\varphi_{m,n}\|_{L^2}\leq C\|(1-\widetilde{\chi}(b_0^{-\frac{1}{\gamma}}x ))\varphi_{m,n}\|_{L^2}
\end{align*}
for another cutoff $\widetilde{\chi}$ such that $\widetilde{\chi}\chi=\widetilde{\chi}$. Therefore, we choose cutoffs $\chi_{(1)},\chi_{(2)},\cdots, \chi_{(N)}\in C_c^{\infty}(\R)$ such that $\chi_{(1)}=\chi$ and $\chi_{(k)}\chi_{(k+1)}=\chi_{(k+1)}$ for all $1\leq k\leq N$ and such that \eqref{induction0} holds by replacing $\chi$ by $\chi_{(k)}$ and 
$$ \|\chi_{(k)}'(b_0^{-\frac{1}{\gamma}}x)\varphi_{m,n}\|_{L^2}\leq C_k\|(1-\chi_{(k+1)}(b_0^{-\frac{1}{\gamma}}x))\varphi_{m,n} )\|_{L^2},\quad k=1,2,\cdots, N-1.
$$
Now since for $\chi_{(N)}$,
\begin{align*}
\|(1-\chi_{(N)}(b_0^{-\frac{1}{\gamma}}x))\varphi_{m,n}\|_{L^2}+\|(1-\chi_{(N)}(b_0^{-\frac{1}{\gamma}}x))h\partial_x\varphi_{m,n}\|\leq Cb_0^{-\frac{1}{\gamma}}h,
\end{align*}
we deduce by induction (in the reverse order) that
$$ \|(1-\chi_{(1)}(b_0^{-\frac{1}{\gamma}}x))\varphi_{m,n}\|_{L^2}+\|(1-\chi_{(1)}(b_0^{-\frac{1}{\gamma}}x))h\partial_x\varphi_{m,n}\|\leq Cb_0^{-\frac{N}{\gamma}}h^N.
$$
This completes the proof of Lemma \ref{ellipticregularity}.
\end{proof}

\subsection{Degenerate regime}\label{s:degenerate}
For $0<h<1$ and $b_0$ fixed once for all thanks to Lemma \ref{ellipticregularity}, we define the semiclassical spectral projector 
$$
\Pi_h^{b_0h}:=\psi(h^2\Delta_\gamma)(\chi_0(h^{\gamma+1}D_y)-\chi_0(b_0hD_y)).
$$
In this subsection, we will show that
\begin{align}\label{degenerateregime}
\|\Pi_h^{b_0h}v_h\|_{L^2(M)}=o(1),\quad h\rightarrow 0. 
\end{align}
We prove it by contradiction. If not, we must have $\|w_h\|_{L^2(M)}\gtrsim 1$ where $w_h=\Pi_h^{b_0h}v_h$. We set $\widetilde{f}=\Pi_h^{b_0h}f$ so that
$$ (h^2\Delta_{\gamma}+1)w_h=\widetilde{f}_h.
$$ 

Let us notice that 
\begin{equation*}
\left| \|h\nabla_\gamma w_h\|^2_{L^2(M)}-\|w_h\|^2_{L^2(M)}\right|\leq \|w_h\|_{L^2(M)}\|(h^2\Delta_\gamma+1)w_h\|_{L^2(M)}
\end{equation*}
where $\nabla_\gamma=(\partial_x,x^{\gamma}\partial_y)$ is the horizontal gradient. This follows from integration by part in the integral $\int w_h(h^2\Delta_\gamma+1)w_h$. We deduce
\begin{equation} \label{e:hnabla}
\|h\nabla_\gamma w_h\|_{L^2(M)}=\|w_h\|^2_{L^2(M)}+o(1).
\end{equation}

The proof of \eqref{degenerateregime} is mainly based on the following commutator relation:
\begin{equation*}
[\Delta_\gamma,x\partial_x+(\gamma+1)y\partial_y]=2\Delta_\gamma.
\end{equation*}
This is an illustration for the positive commutator method, which we shall use again in other parts of the proof. This method dates back at least to \cite[Section 3.5]{Ho71} and has been widely used, for example for proving propagation of singularities for the wave equation.

Note that $y\partial_y$ is not defined globally on $\mathbb{T}_y$. This is why we introduce the following cut-off procedure. Let $\phi\in C^\infty(\T)$ such that $\phi\equiv 1$ on $\T\setminus (a_1,a_2)$, $\text{supp}(\phi')\subset (a_1,a_2)$ and $\phi\equiv 0$ on a strict sub-interval of $I=(a_1,a_2)$. Then, considering $\phi(y)y\partial_y$ on the interval $[\frac{a_1+a_2}{2},\frac{a_1+a_2}{2}+2\pi]$ and then periodizing, we obtain an objet globally defined on $\mathbb{T}$. 

We also set $\chi_{b_0}(x)=\chi(b_0^{-\frac{1}{\gamma}}x)$ (see Lemma \ref{ellipticregularity}). We compute the inner product
\begin{equation*}
C_\gamma:=([h^2\Delta_\gamma+1,\chi_{b_0}(x)\phi(y)(x\partial_x+(\gamma+1)y\partial_y)]w_h,w_h)_{L^2(M)}
\end{equation*}
in two ways. The first way is to expand the bracket and use the self-adjointness of $\Delta_\gamma$:
\begin{equation*}
C_\gamma=(\chi_{b_0}(x)\phi(y)(x\partial_xv_h+(\gamma+1)y\partial_yw_h),\widetilde{f}_h)_{L^2(M)}-(\chi_{b_0}(x) \phi(y)(x\partial_x\widetilde{f}_h+(\gamma+1)y\partial_y\widetilde{f}_h),w_h)_{L^2(M)}.
\end{equation*}
The second way is to use the computation
\begin{align*}
&[h^2\Delta_\gamma+1,\chi_{b_0}(x)\phi(y)(x\partial_x+(\gamma+1)y\partial_y)]\\=&2h^2\chi_{b_0}(x)\phi(y)\Delta_\gamma+h^2(\gamma+1)|x|^{2\gamma}(\phi''(y)y+2\phi'(y))\chi_{b_0}(x)\partial_y\\+&h^2\phi''(y)|x|^{2\gamma}x\chi_{b_0}(x)\partial_x 
+2h^2(\gamma+1)|x|^{2\gamma}y\phi'(y)\chi_{b_0}(x)\partial_y^2+2h^2|x|^{2\gamma}x\chi_{b_0}(x)\phi'(y)\partial^2_{xy}\\
+&h^2\phi(y)(\chi_{b_0}''(x)+2\chi'_{\epsilon}(x)\partial_x )(x\partial_x+(\gamma+1)y\partial_y).
\end{align*}
From the elliptic regularity (Lemma \ref{ellipticregularity}) and \eqref{e:hnabla}, on the supports of $1-\chi_{b_0}(x), \chi_{b_0}'(x),\chi_{b_0}''(x)$, the $L^2$ norm of $w_h$ and $\nabla_{\gamma}w_h$ is of order $O(h^{N})\|w_h\|_{L^2}$ for any $N\in\N$. Then, using integration by part and Young's inequality, we obtain 
\begin{equation*}
C_\gamma=(2\phi(y)h^2\Delta_\gamma w_h,w_h)_{L^2(M)}+O(h)\|h\nabla_\gamma w_h\|^2_{L^2(M)}+O(h)\|w_h\|_{L^2(M)}^2+O(1)\|h\nabla_\gamma w_h\|^2_{L^2(\text{supp}(\phi'))}.
\end{equation*}

Equating the two ways of computing $C_\gamma$ and using integration by parts, we obtain
\begin{align*}
\|\phi(y)^{1/2}h\nabla_\gamma w_h\|^2_{L^2(M)}&\leq O(h)\|h\nabla_\gamma w_h\|^2_{L^2(M)}+O(h)\|w_h\|_{L^2(M)}^2+O(1)\|h\nabla_\gamma w_h\|^2_{L^2(\text{supp}(\phi'))} \\
&+O(1)\|\widetilde{f}_h\|_{L^2(M)}(\|\partial_x w_h\|_{L^2(M)}+\|\partial_y w_h\|_{L^2(M)}).
\end{align*}
First, we notice that we can replace the left hand side simply by $\|h\nabla_\gamma w_h\|^2_{L^2(M)}$ (which is $\gtrsim 1$ thanks to \eqref{e:hnabla}) and the above inequality remains true: this is due to the presence of $O(1)\|h\nabla_\gamma w_h\|^2_{L^2(\text{supp}(\phi'))}$ in the right hand side. Then, for $h$ sufficiently small, we absorb the $O(h)\|h\nabla_\gamma w\|^2_{L^2(M)}$ and the $O(h)\|w_h\|_{L^2(M)}^2$ terms in the left hand side. Finally, we use
\begin{gather}
\|\partial_x w_h\|_{L^2(M)}\leq h^{-1}\|h\nabla_\gamma w_h\|_{L^2(M)}\lesssim h^{-1} \nonumber \\
\|\partial_yw_h\|_{L^2(M)}\leq h^{-(\gamma+1)}\|w_h\|_{L^2(M)} \nonumber \\
\|\widetilde{f}_h\|_{L^2(M)}\leq \|f_h\|_{L^2(M)} =o(h^{\gamma+1}) \nonumber
\end{gather}
where the first line comes from \eqref{e:hnabla}, the second line from Corollary \ref{c:locdy} together with $w_h=\psi(h^2\Delta_\gamma)w_h$, and the third line from Plancherel formula and \eqref{e:seekcontrad}. We obtain
\begin{equation} \label{e:largeinomega}
1\lesssim \|h\nabla_\gamma w_h\|_{L^2(\text{supp}(\phi'))}^2.
\end{equation}

Let us prove that this contradicts \eqref{e:seekcontrad}. Let $\phi_1\in C^\infty(\T_y)$ such that $\phi_1=1$ on $\text{supp}(\phi')$ and $\phi_1=0$ on $\T_y\setminus I$. In particular, together with \eqref{e:largeinomega}, this implies
\begin{equation*} 
1\lesssim \|\phi_1(y)h\nabla_\gamma w_h\|_{L^2(M)}^2.
\end{equation*}
By integration by parts, there holds
\begin{align*}
\|\phi_1(y)h\nabla_\gamma w_h\|_{L^2(M)}^2&=-h^2\int_M w_h (\nabla_\gamma(\phi_1^2)\cdot \nabla_\gamma w_h) dxdy- h^2\int_M \phi_1^2 w_h\Delta_\gamma w_hdxdy\\
&=-h^2\int_M w_h (\nabla_\gamma(\phi_1^2)\cdot \nabla_\gamma w_h) dxdy+ \int_M \phi_1^2 w_h(w_h-\widetilde{f}_h)dxdy
\end{align*}
where in the last line we used the equation of $w_h$. Using \eqref{e:seekcontrad}, \eqref{e:hnabla} and Cauchy-Schwarz inequality, we see that the first term in the last line is $O(h)$. For the second term, we write
\begin{equation*}
\left|\int_M \phi_1^2 w_h(w_h-f_h)dxdy\right|=\|\phi_1 w_h\|_{L^2(M)}^2+o(1),
\end{equation*}
and we note that $$\|\phi_1 w_h\|_{L^2(M)}\leq \|[\phi_1,\Pi_h^{b_0h}]v_h\|_{L^2(M)}+\|\Pi_h^{b_0h}(\phi_1v_h)\|_{L^2(M)} \leq O(h)+ \|v_h\|_{L^2(\omega)}=o(1)$$ 
as $h\rightarrow 0$, by assumption. All in all, we obtain $ \|\phi_1(y)h\nabla_\gamma w_h\|_{L^2(M)}^2=o(1)$, which is a contradiction. This concludes the proof of  \eqref{degenerateregime}.


\subsection{Regime of the geometric control condition} \label{s:gcc}

Let
$$ \Pi_{h,b_0}=\psi(h^2\Delta_G)\chi_0(b_0hD_y)(1-\chi_0(b_0^{-1}hD_y) )
$$
and $z_h=\Pi_{h,b_0}v_h$. In this subsection, we will show that
\begin{align}\label{GCCregime} 
\|z_h\|_{L^2(M)}=o(1),\quad h\rightarrow 0.
\end{align}
We will use a defect-measure based argument as in \cite[Section 5]{BuSun}. It consists in showing that the semi-classical defect measure associated with a subsequence of $(z_h)_{h>0}$ is invariant along the Melrose-Sj\"ostrand flow (corresponding to the principal symbol $p=\xi^2+|x|^{2\gamma}\eta^2$). Then to obtain a contradiction, we just need to check the geometric control condition: there exists $T_0>0$ such that any trajectory of the Melrose-Sj\"ostrand flow enters $\omega$ within time $T_0$; but we recall that only trajectories corresponding to $|\eta|\in (b_0,b_0^{-1})$ are considered here. We omit the standard steps of constructing the semi-classical measure and proving the invariance of the measure\footnote{The argument is the same as in the Baouendi-Grushin-context $\gamma=1$ handled in \cite[Section 5]{BuSun}.}, and only proceed to check the geometric control condition.
 
For the principal symbol
$$ p(x,y;\xi,\eta)=\xi^2+|x|^{2\gamma}\eta^2,\quad \gamma>1,
$$  
the Hamiltonian flow is given by the ODE
\begin{align}\label{hamiltonianflow}
\begin{cases}
& \dot{x}=\partial_{\xi}p=2\xi\\
& \dot{\xi}=-\partial_x p=-2\gamma|x|^{2(\gamma-1)}x\eta^2\\
& \dot{y}=2|x|^{2\gamma}\eta\\
&\dot{\eta}=0.
\end{cases}
\end{align}
Thanks to the integrability of \eqref{hamiltonianflow}, we can define the Melrose-Sj\"ostrand flow associated with the symbol $p$ on the compressed cotangent bundle $^bT^*\overline{M}$.\footnote{In our specific example, the flow in the interior is defined via \eqref{hamiltonianflow}; when it reaches the boundary, the flow is continued directly at diffractive points and by reflection at hyperbolic points. There is no higher order contact in this simple geometry, see \cite[Section 5]{BuSun}.} We will denote by $\varphi_s(\cdot)$ this flow.  
\begin{remark}
The assumption that $\gamma\geq 1$ is used here, since otherwise the coefficients of \eqref{hamiltonianflow} are not Lipschitz and the Cauchy-Lipschitz theorem does not allow us to conclude that its solutions are unique.
\end{remark}
\begin{lemma}\label{dynamical} 
	Assume that $\gamma\geq 1$ and $\omega\subset (-1,1)\times\T$ is a horizontal strip. There exist $T_0>0, c_0>0$, such that for all $\rho_0=(x_0,y_0;\xi_0,\eta_0)$ with $|\eta_0|\in\big(b_0,b_0^{-1}\big)$ and $p(x_0,y_0;\xi_0,\eta_0)=p_0\in \big(\frac{1}{2},2\big)$, there holds
	$$ \frac{1}{T_0}\int_0^{T_0}\mathbf{1}_{\omega}(\varphi_s(\rho_0))ds\geq c_0>0.
	$$
	In particular, the geometric control condition (GCC) holds for $\omega$.
\end{lemma}
\begin{proof}
	It suffices to show that any trajectory $\varphi_s(\rho_0)$ satisfying
	$$ p(\rho_0)=p_0\in\big(\frac{1}{2},2\big),\quad |\eta_0|\in\big(b_0,b_0^{-1}\big)
	$$
	will enter the interior of $\omega$ before some uniform time $T_0>0$. 
	By shifting the $y$ variable we may assume that $y_0=0$. Without loss of generality we can also assume that $\eta_0>0$. Let $\varphi_s(\rho_0)=(x(s),y(s);\xi(s),\eta(s))$. Note that $\eta(s)=\eta_0\neq 0$, so that $x(\cdot)$ is periodic. Moreover, we have the first integrals 
	\begin{equation}\label{e:firstint} p_0=\frac{1}{4}|\dot{x}(s)|^2+|x(s)|^{2\gamma}\eta_0^2,\quad y(s)=2\eta_0\int_0^s|x(s')|^{2\gamma} ds' \; (\text{mod } 2\pi)
	\end{equation}
	 In a nutshell, to show that the flow reaches $\omega$, we first notice that $y(\cdot)$ evolves in a monotone way in $\T$, and that the larger $|x|$ is, the more $y$ varies. Now, if $|x|$ remains too small, then \eqref{e:firstint} gives that $|\dot{x}|\sim 2\sqrt{p_0}$, which implies that $|x|$ cannot remain too small, thus a contradiction.
	
	To put it into a rigorous form, consider the interval $J_{\delta}=(-\delta,\delta)$ (for the $x$ variable) for $0<\delta\ll 1$. For $\delta>0$ sufficiently small (not depending on $|\eta_0|\in (b_0,b_0^{-1})$) and if $x(s)\in J_\delta$, using \eqref{e:firstint}, we have
	$ |\dot{x}(s)|\geq \sqrt{p_0}$. Therefore, following the flow, it takes a time at most $\tau_0:=\frac{2\delta}{\sqrt{p_0}}$ to leave the regionl $J_{\delta}\times \T$. \\
	Let us fix $s_0$ such that $x(s_0)\in J_\delta$ (if it does not exist, we are done thanks to the second relation in \eqref{e:firstint}). We know that there exists $s_0\leq s_1\leq s_0+\tau_0$ such that $|x(s_1)|=\delta$. We consider the minimal time $s_2\geq s_1$ such that $|x(s_2)|=\frac{\delta}{2}$. Since $\|\dot{x}\|_{\infty}\leq 2\sqrt{p_0}$ (thanks to \eqref{e:firstint}), we know that $s_2\geq s_3:=s_1+\frac{\delta}{4\sqrt{p_0}}$. Finally,
	\begin{equation*}
	y(s_3)-y(s_0)=2\eta_0\int_{s_0}^{s_3}|x(s')|^{2\gamma}ds'\geq 2b_0\int_{s_1}^{s_3}\left(\frac{\delta}{2}\right)^{2\gamma}ds'=\frac{b_0\delta}{2\sqrt{p_0}}\left(\frac{\delta}{2}\right)^{2\gamma}.
	\end{equation*}
	In other words, in any case, $y$ increases of at least $\frac{b_0\delta}{2\sqrt{p_0}}\left(\frac{\delta}{2}\right)^{2\gamma}$ within any time period of length $\tau_0+\frac{\delta}{4\sqrt{p_0}}\leq \frac{3\delta}{\sqrt{p_0}}$. Hence, the result holds for some $T_0$ of order $\delta^{-2\gamma}$.
\end{proof}



\subsection{Horizontal propagation regime I} \label{s:hor1}

		Now we treat the regime $|D_y|\leq b_0h^{-1}$. We set
		$ \kappa_h:=\psi(h^2\Delta_{\gamma})\chi_0(b_0^{-1}hD_y)v_h.
		$
		To finish the proof of Theorem \ref{t:resgamma},  it remains to show that
		\begin{align}\label{regimehorizontal1} 
		\|\kappa_h\|_{L^2(M)}=o(1),\quad h\rightarrow 0.
		\end{align}
	Let $\mu$ be a semi-classical measure associated to a subsequence of $(\kappa_h)_{h>0}$. Since it is invariant along the Hamiltonian flow associated with $p=\xi^2+|x|^{2\gamma}\eta^2$ subject to the reflection and diffraction at the boundary. Since $\mu\mathbf{1}_{\omega}=0$ and $\omega$ is a horizontal strip (or union of horizontal strips), we deduce that the only possible place where the defect measure concentrates is the set $\{\eta=0\}$ on which the trajectories are horizontal. To exclude this possibility, we need to  decompose $|D_y|$ in a finer way. For some small parameter $\epsilon>0$ to be chosen later, we let
$$ \kappa_h^{\epsilon}=(1-\chi_0(h^{\epsilon}D_y))\kappa_h,\quad \kappa_{h,\epsilon}=\chi_0(h^{\epsilon}D_y)\kappa_h.
$$
Our goal of this subsection  is to show that
\begin{align}\label{regimehorizontal1.1}
 \|\kappa_h^{\epsilon}\|_{L^2(M)}=o(1),\quad h\rightarrow 0
\end{align} 
We use the positive commutator method (already used in Section \ref{s:degenerate}) with the relation
\begin{align*}
[h^2\Delta_{\gamma}+1,\phi(y)y\partial_y]=&2\phi(y)|x|^{2\gamma}(h\partial_y)^2+2y\phi'(y)|x|^{2\gamma}(h\partial_y)^2\\
+& h^2\phi''(y)y|x|^{2\gamma}\partial_y+2h^2\phi'(y)|x|^{2\gamma}\partial_y,
\end{align*}	
where $\phi$ has been introduced in Section \ref{s:degenerate}.
As in Section \ref{s:degenerate}, we compute the inner product $([h^2\Delta_{\gamma}+1,\phi(y)y\partial_y]\kappa_h^{\epsilon},\kappa_h^{\epsilon})_{L^2(M)}$ in two ways, and using Cauchy-Schwarz, it gives
\begin{align*}
\|\phi(y)^{1/2}h|x|^{\gamma}\partial_y\kappa_h^{\epsilon}\|_{L^2(M)}^2\leq & Ch\||x|^{\gamma}h\partial_y\kappa_{\epsilon}^h\|_{L^2(M)}\|\kappa_h^{\epsilon}\|_{L^2(M)}+Ch^2\|\kappa_h^{\epsilon}\|_{L^2(M)}^2\\+&C\|\phi'(y)^{1/2}|x|^{\gamma}h\partial_y\kappa_h^{\epsilon}\|_{L^2(M)}^2
+Ch^{-1}\|f_h\|_{L^2(M)}\|h\partial_y\kappa_h^{\epsilon}\|_{L^2(M)}. 
\end{align*}	
Using Young's inequality, we deduce that for any $\delta>0$, for any sufficiently small $h>0$, 
\begin{align}
\|\phi(y)^{1/2}|x|^{\gamma}h\partial_y\kappa_h^{\epsilon}\|_{L^2(M)}^2\leq & \delta \|h\partial_y\kappa_h^{\epsilon}\|_{L^2(M)}^2+C\||x|^{\gamma}h\partial_y\kappa_h^{\epsilon}\|_{L^2(\mathrm{supp}(\phi'))}^2\notag \\
+&C(\delta)h^{-2}\|f_h\|_{L^2(M)}^2+C(\delta)h^2\|\kappa_h^{\epsilon}\|_{L^2(M)}^2 \notag
\end{align}
and therefore, using the $\||x|^{\gamma}h\partial_y\kappa_h^{\epsilon}\|_{L^2(\mathrm{supp}(\phi'))}$ term in the right hand side, we obtain
\begin{equation}\label{eq:horizontal1}	
\begin{aligned}
\||x|^{\gamma}h\partial_y\kappa_h^{\epsilon}\|_{L^2(M)}^2\leq  \delta \|h\partial_y\kappa_h^{\epsilon}\|_{L^2(M)}^2+C\||x|^{\gamma}h\partial_y\kappa_h^{\epsilon}\|_{L^2(\mathrm{supp}(\phi'))}^2 \\
+C(\delta)h^{-2}\|f_h\|_{L^2(M)}^2+C(\delta)h^2\|\kappa_h^{\epsilon}\|_{L^2(M)}^2.
\end{aligned}
\end{equation}
		
We need the following lemma, which roughly states that in the horizontal regime, the mass cannot concentrate on $x=0$:
\begin{lemma}\label{propagation:horizontal}
We have 
$$ \|\partial_y\kappa_h^{\epsilon}\|_{L^2(M)}\leq C\||x|^{\gamma}\partial_y\kappa_h^{\epsilon}\|_{L^2(M)}+o(1),
$$
as $h\rightarrow 0$.
\end{lemma}
Let us postpone the proof of Lemma \ref{propagation:horizontal} for the moment and proceed to finish the proof of \eqref{regimehorizontal1.1}. Thanks to \eqref{eq:horizontal1} and Lemma \ref{propagation:horizontal}, by choosing $\delta$ small enough, we have
\begin{align*}
\||x|^{\gamma}\partial_y\kappa_h^{\epsilon}\|_{L^2(M)}^2\leq C\||x|^{\gamma}\partial_y\kappa_h^{\epsilon} \|_{L^2(\mathrm{supp}(\phi') )}^2+C(\delta)o(h^{2(\gamma-1)})+C(\delta)\|\kappa_h^{\epsilon}\|_{L^2(M)}^2+o(1).
\end{align*}
Applying Lemma \ref{propagation:horizontal} again and plugging into the inequality above, we have
\begin{align*}
\|\partial_y\kappa_h^{\epsilon}\|_{L^2(M)}^2+
\||x|^{\gamma}\partial_y\kappa_h^{\epsilon}\|_{L^2(M)}^2\leq C\||x|^{\gamma}\partial_y\kappa_h^{\epsilon} \|_{L^2(\mathrm{supp}(\phi') )}^2+C(\delta)\|\kappa_h^{\epsilon}\|_{L^2(M)}^2+o(1).
\end{align*}
Now since  $\mathcal{F}_y(\kappa_h^{\epsilon})(x,n)=0$, for all $|n|\leq h^{-\epsilon}$, by definition of $\kappa_h^{\epsilon}$, we have $$\|\kappa_h^{\epsilon}\|_{L^2(M)}^2\leq h^{2\epsilon}\|\partial_y\kappa_h^{\epsilon}\|_{L^2(M)}^2,$$   
hence, we have
$$ \|\partial_y\kappa_h^{\epsilon}\|_{L^2(M)}^2+
\||x|^{\gamma}\partial_y\kappa_h^{\epsilon}\|_{L^2(M)}^2\leq C\||x|^{\gamma}\partial_y\kappa_h^{\epsilon} \|_{L^2(\mathrm{supp}(\phi') )}^2+o(1).
$$
Now, we proceed as in Section \ref{s:degenerate}: we insert a smooth cutoff $\phi_1(y)$ such that $\mathrm{supp}(\phi_1)\subset \omega $ and $\phi_1(y)\equiv 1$ on supp$(\phi')$. Hence $\||x|^{\gamma}\partial_y\kappa_h^{\epsilon} \|_{L^2(\mathrm{supp}(\phi') )}^2\leq \||x|^{\gamma}\phi_1(y)\partial_y\kappa_h^{\epsilon} \|^2$. Then we can write
$$ |x|^{\gamma}\phi_1(y)h\partial_y\kappa_h^{\epsilon}=|x|^{\gamma}h\partial_y\psi(h^2\Delta_{\gamma})\chi_0(b_0^{-1}hD_y)(\phi_1(y)v_h )+O_{L^2(M)}(h),
$$
where the second term on the r.h.s. comes from the commutator. Therefore,
\begin{align*}
\|\kappa_h^{\epsilon}\|_{L^2(M)}^2\leq h^{2\epsilon}\|\partial_y\kappa_h^{\epsilon}\|_{L^2(M)}^2\leq o(h^{2\epsilon})+Ch^{2\epsilon}\|\phi_1(y)v_h\|_{L^2(M)}^2+O(h^{2\epsilon}).
\end{align*}
Since $\text{supp}(\phi_1)\subset \omega$, there holds $\|\phi_1(y)v_h\|_{L^2(M)}=o(1)$. The proof of \eqref{regimehorizontal1.1} is complete.

It remains to prove Lemma \ref{propagation:horizontal}:
\begin{proof}[Proof of Lemma \ref{propagation:horizontal}]
This is a variant of horizontal propagation estimates in the spirit of Lemma 6.2 in \cite{BuSun}. However, due to the absent of the time variable, here we need a slightly different argument. The main idea is to use propagation arguments in the horizontal direction in order to ``get out'' from the singular region $x=0$.

Let $z_h=\partial_y\kappa_h^{\epsilon}$. Since $\partial_y$ commutes with $h^2\Delta_{\gamma}+1$, $z_h$ satisfies the equation
$$ (h^2\Delta_{\gamma}+1)z_h=g_h=o_{L^2}(h^{\gamma}),
$$ 
where $g_h=(1-\chi_0(h^{\epsilon}D_y))\chi_0(b_0^{-1}hD_y)\partial_yf_h$.  Let us show that for some $r_0\in (0,\frac{1}{2})$,
\begin{align*}
 \|z_h\|_{L^2(|x|\leq 2r_0)}\leq C(r_0)\|z_h\|_{L^2(r_0<|x|<1)}+o(1)
\end{align*}
as $h\rightarrow 0$, which is sufficient for proving Lemma \ref{propagation:horizontal}. We choose $\psi^{\pm}\in C_c^{\infty}(\R)$ such that $$\psi^{\pm}(\xi)=\begin{cases}
1,\quad &\text{ if }\quad \frac34\sqrt{\frac12-(8C_1)^{\gamma+1}b_0^2 } \leq \pm \xi\leq 2\sqrt{2};\\
0,\quad &\text{ if }\quad  |\xi|>3 \text{ or } |\xi|<\frac{1}{2}\sqrt{\frac{1}{2}-(8C_1)^{\gamma+1}b_0^2}.
\end{cases}
$$ 
Let $\chi\in C_c^{\infty}((0,1))$ such that $\chi(x)=1$ if $|x|\leq r_0$ and $\chi(x)=0$ if $|x|>3r_0/2$. From the localization property of $z_h$, we know that 
\begin{align*}
\mathrm{WF}_h(z_h)&\subset \big\{(x,y;\xi,\eta): p=\xi^2+|x|^{2\gamma}\eta^2\in(\frac{1}{2},2), |\eta|\leq (8C_1)^{\frac{\gamma+1}{2}}b_0 \big\}\\
&\subset \{(x,y,\xi,\eta): \xi\in\mathrm{supp}(\psi^+)\cup\mathrm{supp}(\psi^-) \},
\end{align*} thus it suffices to estimate $\|\chi(x)\psi^{\pm}(hD_x)z_h\|_{L^2(M)}$ and by symmetry we only need to estimate
$\|\chi(x)\psi^{+}(hD_x)z_h\|_{L^2(M)}$.
Moreover, by our choice of $\psi^{\pm}$,
$$ \mathrm{WF}_h(z_h)\cap\{(x,y;\xi,\eta): \xi\in \mathrm{supp}((\psi^{\pm})')\}=\emptyset.
$$

Note that for any (semi-classical) pseudo-differential operator $\mathrm{Op}_h(a)$, compactly supported in the interior of $M$, we have
\begin{align}\label{commutator}
 \frac{1}{ih}\big([\mathrm{Op}_h(a),h^2\Delta_{\gamma}+1]z_h,z_h\big)_{L^2(M)}=o(h^{\gamma-1})=o(1),
\end{align}
thanks to the equation of $z_h$. Now we consider a specific pseudo-differential operator $\mathrm{Op}_h(a^{\pm})$ with principal symbol $\chi^2(x)\sin\big(\frac{\pi x}{4r_0}\big)(\psi^{\pm}(\xi))^2$. By symbolic calculus, 
$$ \frac{1}{ih}[\mathrm{Op}_h(a^{+}),h^2\Delta_{\gamma}+1]=\mathrm{Op}_h(\{\xi^2+|x|^{2\gamma}\eta^2,\chi^2(x)\sin(\frac{\pi x}{4r_0})(\psi^{\pm}(\xi))^2 \})+\mathcal{O}_{L^2\rightarrow L^2}(h).
$$
We compute
\begin{align*}
 &\{\xi^2+|x|^{2\gamma}\eta^2,\chi^2(x)\sin(\frac{\pi x}{4r_0})(\psi^{+}(\xi))^2 \}\\=&2\xi(\psi^{+}(\xi))^2\cdot\frac{\pi}{4r_0}\chi^2(x)\cos\big(\frac{\pi x}{4r_0}\big)+4\xi(\psi^{+}(\xi))^2\chi(x)\chi'(x)\sin\big(\frac{\pi x}{4r_0}\big)\\
 -&4\gamma|x|^{2\gamma-2}x\eta^2\psi^+(\xi)(\psi^{+})'(\xi)\chi^2(x)\sin\big(\frac{\pi x}{4r_0}\big).
\end{align*}
Let
$$ a_1=2\xi(\psi^{+}(\xi))^2\cdot\frac{\pi}{4r_0}\chi^2(x)\cos\big(\frac{\pi x}{4r_0}\big),\quad a_2=4\xi(\psi^{+}(\xi))^2\chi(x)\chi'(x)\sin\big(\frac{\pi x}{4r_0}\big)
$$
and $a_3=-4\gamma|x|^{2\gamma-2}x\eta^2\psi^{+}(\xi)(\psi^{+})'(\xi)\chi^2(x)\sin\big(\frac{\pi x}{4r_0}\big)$. From the property of $\mathrm{WF}_h(z_h)$, we have $(\mathrm{Op}_h(a_3)z_h,z_h)_{L^2(M)}=O(h^N)$, for any $N\in\N$. From the support property of $a_2$, we have
$$ |(\mathrm{Op}_h(a_2)z_h,z_h)_{L^2(M)}|\leq C\|z_h\|_{L^2(r_0<|x|<1)}^2.
$$
Thus from \eqref{commutator}, we have
\begin{align}\label{Op:a1}
 (\mathrm{Op}_h(a_1)z_h,z_h)_{L^2(M)}\leq o(1)+C\|z_h\|_{L^2(r_0<|x|<1)}^2.
\end{align}
Since $a_1\geq c_0$ for some uniform constant $c_0>0$ on $|x|\leq 3r_0/2$, we can decompose $a_1=a_1^{(0)}+a_1^{(1)}$ where $a_1^{(0)}\geq c_0\chi(x)^2(\psi^{+}(\xi))^2$ and $\mathrm{supp}(a_1^{(1)})\subset \{|x|>\frac{3r_0}{2} \}$. Using the sharp G{\aa}rding inequality, we have
$$ (\mathrm{Op}_h(a_1^{(0)})z_h,z_h)_{L^2(M)}\geq c_0(\mathrm{Op}_h(\chi(x)^2(\psi^+(\xi))^2)z_h,z_h)_{L^2(M)}-Ch\|z_h\|_{L^2(M)}^2.
$$
Together with \eqref{Op:a1}, this yields
$$ \|\chi_0(x)\psi(hD_x)^{+}z_h\|_{L^2}^2\leq o(1)+C\|z_h\|_{L^2(r_0<|x|<1)}.
$$
The proof of Lemma \ref{propagation:horizontal} is now complete.
\end{proof}

 
 \subsection{Horizontal propagation regime II}	\label{s:hor2}
 To finish the proof of \eqref{regimehorizontal1} (and hence that of Theorem \ref{t:resgamma}), it remains to show that
 \begin{align}\label{regimehorizontal2} 
 \|\kappa_{h,\epsilon}\|_{L^2(M)}=o(1),\quad h\rightarrow 0,
 \end{align}	
 where $\kappa_{h,\epsilon}=\chi_0(h^{\epsilon}D_y)\kappa_h=\psi(h^2\Delta_{\gamma})\chi_0(h^{\epsilon}D_y)\kappa_h$ with small parameter $\epsilon>0$ to be fixed later.
 In this subsection, we prove the following result which, combined with \eqref{e:seekcontrad}, directly yields \eqref{regimehorizontal2}:
 \begin{prop}\label{Observation-Lowfrequency}
 	There exist $C>0$, $h_0>0$, $\eps_0>0$ such that for all $0<h<h_0$ and $0<\epsilon<\eps_0$, we have
 	\begin{equation*}
 	\|\kappa_{h,\epsilon}\|_{L^2(M)}\leq C\|\kappa_{h,\epsilon}\|_{L^2(\omega)}+Ch^{-2}\|(h^2\Delta_{\gamma}+1)\kappa_{h,\epsilon}\|_{L^2(M)}+Ch^{1-2\epsilon}\|v_h\|_{L^2(M)}.
 	\end{equation*}
 \end{prop}
 We follow the normal form method as in \cite[Section 7]{BuSun}, originally inspired by the work \cite{BZ04}. The key point is to search for a microlocal transformation
 $$ w=(1+hQD_y^2)v
 $$
 for some suitable semi-classical pseudo-differential operator $Q=q(x,hD_x)$, such that the conjugated equation (satisfed by $w$) is
 $$ h^2\partial_x^2w+h^2M\partial_y^2w+w=\textrm{ errors},
 $$  
 where 
 $$M = \frac 1 2 \int_{-1}^1 |x|^{2\gamma}dx$$
 is the mean value of $|x|^{2\gamma}$. Roughly speaking, this normal form method puts into a rigorous form the intuition that in the horizontal propagation regime, the vector field $|x|^\gamma\partial_y$ acts as if it were averaged along horizontal trajectories.
 
 Then we will be able to use the following theorem:
 \begin{prop}[
 	\cite{AL14},\cite{BZ04},\cite{AM14}]
 	\label{ob:Schrodinger}
 	Let $\Delta_M=\partial_x^2+M\partial_y^2$. Then for any non-empty open set $\omega_0\subset\T^2$, we have
 	\begin{equation*}
 	\|u\|_{L^2(\T^2)}\leq C\|u\|_{L^2(\omega_0)}+Ch^{-2}\|(h^2\Delta_M+1)u\|_{L^2(\T^2)}.
 	\end{equation*}
 \end{prop}
 
 However,  dealing with Dirichlet boundary value problem induces difficulties and consequently,  we prefered to extend the analysis to the periodic setting. First we introduce several notations. Let
 $$ \widetilde{\mathbb{T}}:=[-1,3]\slash\{-1,3\} \quad \textrm{ and }\quad \widetilde{\mathbb{T}}^2:=\widetilde{\mathbb{T}}_x\times \mathbb{T}_y.
 $$
 Define
 $$ a(x)=|x|^{\gamma}, \textrm{ if }|x|\leq 1 \textrm{ and } a(x)=|2-x|^{\gamma}, \textrm{ if } 1\leq x\leq 3,
 $$
 and the operator
 $$ P_a:=\partial_x^2+a(x)^2\partial_y^2.
 $$ 
 Note that $a(x)$ and $a(x)^2$ are Lipschitz functions on $\widetilde{\mathbb{T}}$.
Let
 $$ H_a^k(\widetilde{\mathbb{T}}^2):=\{f\in \mathcal{D}'(\widetilde{\mathbb{T}}^2): P_a^jf\in L^2(\widetilde{\mathbb{T}}^2),\forall 0\leq j\leq k \}
 $$
 the associated function spaces and the domain of $P_a$ is $D(P_a)=H_{a}^2(\widetilde{\mathbb{T}}^2).$ 
Recall that $D(\Delta_{\gamma})=H_{\gamma,0}^1(M)\cap H_{\gamma}^2(M)$.
 Consider the extension map:
 $$ \iota_1:D(\Delta_{\gamma})\rightarrow D(P_a), \quad f\mapsto \widetilde{f}, 
 $$
 with
 $$ \widetilde{f}(x,y)=f(x,y),\textrm{ if } |x|\leq 1,\textrm{ and } \widetilde{f}(x,y)=-f(2-x,y),\textrm{ if }1\leq x\leq 3.
 $$
 The mapping $\iota_1$ is the odd extension with respect to $x=1$. Note that for $f\in C^{\infty}(\overline{M} )$, we have
 $$ \partial_xf|_{x=1-}=\partial_x(\iota_1f)|_{x=1+}.
 $$
 Recall the following lemmas from \cite[Section 7]{BuSun}:
 \begin{lemma}[\cite{BuSun}]\label{embedding1}
 	The extension map $\iota_1: D(\Delta_{\gamma})\rightarrow D(P_a)$ is continuous. Moreover, for all $f\in D(\Delta_{\gamma})$, $\|\iota_1f\|_{L^2(\widetilde{\T}^2)}=\sqrt{2}\|f\|_{L^2(\Omega)}$.
 \end{lemma}
 Note that this result was only proved for $\gamma=1$ in \cite[Section 7]{BuSun}, but the proof given there works without any modification for general $\gamma\geq 1$. 
 
 \begin{lemma}[\cite{BuSun}]\label{commutation}
 	Let $S_1,S_2$ be two self-ajoint operators on Banach spaces $E_1, E_2$ with domains $D(S_1), D(S_2)$ respectively. Assume that $j: D(S_1)\rightarrow D(S_2)$ is a continuous embedding and that there holds $j\circ S_1=S_2\circ j$. Then, for any Schwartz function $g\in \mathcal{S}(\R)$, we have
 	$$ j\circ g(S_1)=g(S_2)\circ j
 	$$   	
 \end{lemma}
 
 Lemma \ref{commutation}  ensures the preservation of the spectral localization property by odd extension procedure.  We deduce from Lemma \ref{commutation} that for any Schwartz function $g:\R\rightarrow \C$,
 \begin{equation*}
 \iota_1\circ g(h^2\Delta_{\gamma})=g(h^2P_a)\circ\iota_1. 
 \end{equation*}
 Consequently, we have the following lemma, reducing the proof of Proposition \ref{Observation-Lowfrequency} to the observability of the extended solutions:
 \begin{lemma}\label{observability:extension}
 	Let $\phi_1(y)$ be a smooth function which is supported in $\omega$. Assume that there exist $h_0,\eps_0>0$ such that for any $0<h<h_0, 0<\epsilon<\eps_0$, the following observability holds for all  $\widetilde{v}\in L^2(\widetilde{\T}^2)$:
 	\begin{align}\label{ob:extension}
 	\|\psi(h^2P_a)\chi_0(h^{\epsilon}D_y)\widetilde{v}\|_{L^2(\widetilde{\T}^2)}^2\leq & C\|(h^2P_a+1)\psi(h^2P_a)\chi_0(h^{\epsilon}D_y)\widetilde{v}\|_{L^2(\widetilde{\T}^2)}^2 \notag\\
 	+ &C\|\phi_1(y)\psi(h^2P_a)\chi_0(h^{\epsilon}D_y)\widetilde{v}(t)\|_{L^2(\widetilde{\T}^2)}^2dt+Ch\|\widetilde{v}\|_{L^2(\widetilde{\T}^2)}^2.
 	\end{align}
 	Then Proposition \ref{Observation-Lowfrequency} is true. More precisely, with the same constant $C>0$, for all $0<h<h_0$, $0<\epsilon<\eps_0$, the resolvent estimate
 	\begin{align*}
 	\|\psi(h^2\Delta_{\gamma})\chi_0(h^{\epsilon}D_y)v\|_{L^2(M)}^2\leq & C\|(h^2\Delta_{\gamma}+1)v\|_{L^2(M)}^2
 	\notag\\
 	+ &C\|\phi_1(y)\psi(h^2\Delta_{\gamma})\chi_0(h^{\epsilon}D_y)v\|_{L^2(M)}^2+Ch\|v\|_{L^2(M)}^2
 	\end{align*}
 	holds for all $v\in L^2(M)$.
 \end{lemma}
 The proof of Lemma \ref{observability:extension} is straightforward and we omit the detail.
 \begin{rem}
 	Since the extension operation is done for the $x$-variable, we keep the notation $\phi_1(y)$ for the extension of this function.
 \end{rem}
 Before proving \eqref{ob:extension}, we need a lemma which, modulo errors, allows us to replace the operator $\psi(h^2P_a)\chi_0(h^{\epsilon}D_y)$ by $\psi_1(hD_x)\chi_0(h^{\epsilon}D_y)$.
 \begin{lemma}\label{changing-localization}
 	Let $\psi_1\in C_c^{\infty}(\frac{1}{4}<|\xi|<4)$ such that $\psi_1=1$ on supp$(\psi)$. Then, as a bounded operator on $L^2(\widetilde{\T}^2)$, we have
 	$$ \left(1-\psi_1(hD_x) \right)\psi(h^2P_a)\chi_0(h^{\epsilon}D_y)=O_{L^2\rightarrow L^2}(h^{2-2\epsilon}).
 	$$ 	
 \end{lemma}
 \begin{proof}
 As $D_y$ commutes with $D_x$ and $P_a$, by Plancherel, it suffices to show that, uniformly in $|n|\leq (8C_1)^{\frac{1}{\gamma+1}}h^{-\epsilon}$,
 $$ (1-\psi_1(hD_x))\psi(h^2\mathcal{L}_n)=O_{L^2\rightarrow L^2}(h^{2(1-\epsilon)}),
 $$
 where $\mathcal{L}_n=-\partial_x^2+n^2a(x)^2$. The key point here is that $(1-\psi_1(\xi))\psi(\xi)=0$. We will make use of the Helffer-Sj\"ostrand formula (see \cite{DS99} and \cite{BGT04}) :
 $$ \psi(h^2\mathcal{L}_n)=\frac{1}{2\pi i}\int_{\C}\ov{\partial}\widetilde{\psi}(z)(z-h^2\mathcal{L}_n)^{-1}dz\wedge d\ov{z},
 $$
where $\widetilde{\psi}(z)$ is an almost analytic extension of $\psi$, for example
$$ \widetilde{\psi}(z):=\chi(\Im z)\cdot \sum_{n=0}^{N+1}\frac{\psi^{(n)}(\Re z) }{n!} (i\Im z)^n,\quad N\geq 2.
$$
Note that as an operator-valued meromorphic function, we have
$$ (z-h^2\mathcal{L}_n)^{-1}=(z-h^2D_x^2)^{-1}+h^2n^2(z-h^2D_x^2)^{-1}a(x)^2(z-h^2\mathcal{L}_n)^{-1},
$$
we obtain that
\begin{align*}
(1-\psi_1(hD_x))\psi(h^2\mathcal{L}_n)=\frac{h^2n^2}{2\pi i}(1-\psi_1(hD_x))\int_{\C} \ov{\partial}\widetilde{\psi}(z)(z-h^2D_x)^{-1}a(x)^2(z-h^2\mathcal{L}_n)^{-1}dz\wedge d\ov{z},
\end{align*}
where we used the Cauchy integral formula
$$ \psi(h^2D_x)=\frac{1}{2\pi i}\int_{\C}\ov{\partial}\widetilde{\psi}(z)(z-h^2D_x^2)^{-1}dz\wedge d\ov{z}
$$
and $(1-\psi_1(hD_x))\psi(hD_x)=0$.
Using the fact that $|\ov{\partial}\widetilde{\psi}(z)|\leq C_N|\Im z|^N\chi(\Im z)$ and $\|(z-P)^{-1}\|\leq |\Im z|^{-1}$ for any self-adjoint operator $P$, we deduce that
$$ \|(1-\psi_1(hD_x))\psi(h^2\mathcal{L}_n)\|\leq Ch^{2(1-\epsilon)}.
$$
This completes the proof of Lemma \ref{changing-localization}.
 \end{proof}
 
 \begin{proof}[Proof of Proposition \ref{Observation-Lowfrequency}]
 	From Lemma \ref{observability:extension}, it is sufficient to prove  \eqref{ob:extension}. With a little abuse of notation, we denote by $v_0=\widetilde{\kappa}_{h,\epsilon}$ the extension of $\kappa_{h,\epsilon}$, which verifies $v_0=\psi(h^2P_a)\chi_0(h^{\epsilon}D_y)v_0$. We are now in the periodic setting. Yet, we should pay an extra attention to the fact that $P_a=\partial_x^2+a(x)^2\partial_y^2$ is a hypoelliptic operator with only Lipschitz coefficient. More precisely, $a\in \mathrm{Lip}(\widetilde{\T}^2)$ which is not $C^1$ when $\gamma=1$ at $x=1$.

 	Modulo an error $O_{L^2}(h^{2-2\epsilon})\|v_0\|_{L^2(\widetilde{\T}^2)}$, we may assume that $v_0=\psi_1(hD_x)\chi_0(h^{\epsilon}D_y)v_0$. Note that
 	$$ (h^2P_a+1)v_0=f_0:=\iota((h^2\Delta_{\gamma}+1)\kappa_{h,\epsilon}).
 	$$
 	Now we search for the function
 	$$ w_0=(1-hQ\partial_y^2)v_0
 	$$
 	with an $h$-pseudo-differential operator $Q$ acting only on $x$, to be chosen later. Let $M=\frac{1}{4}\int_{-1}^3 a(x)^2dx= \frac 1 2 \int_{-1}^1 |x|^{2\gamma} dx$ be the average of $a(x)^2$ along the horizontal trajectory $y=\mathrm{const.}$ Using the equation $(h^2P_a+1)v_0=f_0$, we have
 	\begin{equation*}
 	\begin{split}
 	(h^2\partial_x^2+Mh^2\partial_y^2)w_0+w_0=&(1-hQ\partial_y^2)(h^2\Delta_{\gamma}+1)v_0+(1-hQ\partial_y^2)(M-a(x)^2)h^2\partial_y^2v_0\\
 	-&\frac{1}{h}[h^2\partial_x^2,Q]h^2\partial_y^2v_0\\
 	=& (1-hQ\partial_y^2)f_0+\big(M-a(x)^2-\frac{1}{h}[h^2\partial_x^2,Q]\big)h^2\partial_y^2v_0\\
 	-&hQ\partial_y^2(M-a(x)^2)h^2\partial_y^2v_0
 	\end{split}
 	\end{equation*}
 	Take $\psi_2\in C_c^{\infty}(1/8\leq |\xi|\leq 8)$, such that $\psi_2\psi_1=\psi_1$. We define the operator
 	$$ Q=\frac{1}{2 i}\left(\int_{-1}^x(M-a(z)^2 )dz\right)(hD_x)^{-1}\psi_2(hD_x),
 	$$ 
 	and set $b(x)=\frac{1}{2 i}\int_{-1}^x(M-a(z)^2 )dz$, $m(hD_x)=(hD_x)^{-1}\psi_2(hD_x)$. Since $a(x)^2-M$ has zero average, the function $b$ is well-defined as a periodic function in the space $C^1(\widetilde{\T})\cap W^{2,\infty}(\widetilde{\T})$. From a direct calculation, we have
 	$$ h[\partial_x^2,Q]=2ib'(x)m(hD_x)hD_x+i[hD_x,b'(x)]m(hD_x).
 	$$
 	Note that $[hD_x,b'(x)]=-ihb''(x)$, and $b''\in L^{\infty}(\widetilde{\T})$, thus
 	$$ \|\big(M-a(x)^2-\frac{1}{h}[h^2\partial_x^2,Q]\big)h^2\partial_y^2v_0\|_{L^2(\widetilde{\T}^2)}=O(h^{3-2\epsilon})\|v_0\|_{L^2(\widetilde{\T}^2)}.
 	$$
 	Therefore,
 	$$ \|(h^2\Delta_M+1)w_0\|_{L^2(\widetilde{\T}^2)}\leq C\|f_0\|_{L^2(\widetilde{\T}^2)}+O(h^{3-4\epsilon})\|v_0\|_{L^2(\widetilde{\T}^2)}.
 	$$
 	where $\Delta_M=\partial_x^2+M\partial_y^2$.
 	Applying Proposition \ref{ob:Schrodinger}, we obtain that
 	\begin{equation*}
 	\begin{split}
 	\|w_0\|_{L^2(\widetilde{\T}^2)}\leq C\|\phi_1(y)w_0\|_{L^2(\widetilde{\T}^2)}+Ch^{-2}\|f_0\|_{L^2(\widetilde{\T}^2)}+Ch^{1-4\epsilon}\|v_0\|_{L^2(\widetilde{T}^2)}.
 	\end{split}
 	\end{equation*}
 	Since $w_0=v_0+O_{L^2(\widetilde{\T}^2)}(h^{1-2\epsilon})\|v_0\|_{L^2(\widetilde{\T}^2)}$ and supp$(\phi_1)\subset \omega$, the proof of Proposition \ref{Observation-Lowfrequency} is now complete.
 \end{proof}
 
 \begin{proof}[End of the proof of Theorem \ref{t:resgamma}]
We choose $\eps$ as in Proposition \ref{Observation-Lowfrequency}. Combining \eqref{e:locvcoercivity}, \eqref{degenerateregime}, \eqref{GCCregime}, \eqref{regimehorizontal1} and \eqref{regimehorizontal2}, we obtain $\|v_h\|_{L^2(M)}=o(1)$, which contradicts \eqref{e:seekcontrad} and proves Theorem \ref{t:resgamma}.
 \end{proof}

\section{Theorem \ref{t:main}: proofs of observability} \label{s:proof12}

\subsection{Localized observability} \label{s:locobssection}

In this section, we prove Point (1) and one part of Point (2) of Theorem \ref{t:main}, namely that observability holds for sufficiently large time in case $s=\frac{\gamma+1}{2}$. The proofs of these two results are both based on the resolvent estimate given by Theorem \ref{t:resgamma}.

In general, we have the following abstract theorem: 
\begin{thm}[\cite{BZ04}]\label{abstract}
	Let $P(h)$ be self-adjoint on some Hilbert space $\mathcal{H}$ with densely defined domain $\mathcal{D}$ and $A(h):\mathcal{D}\rightarrow \mathcal{H}$ be bounded. Fix $\chi_0\in C_c^{\infty}((-b,-a))$.  Assume that uniformly for $\tau\in I=[-b,-a]\subset \R$, we have the following resolvent inequality
	$$ \|u\|_{\mathcal{H}}\leq \frac{G(h)}{h}\|(P(h)+\tau)u\|_{\mathcal{H}}+g(h)\|A(h)u\|_{\mathcal{H}}
	$$
	for some $1\leq G(h)\leq O(h^{-N_0})$.
	Then there exist constants $C_0,c_0,h_0>0$, such that for every $T(h)$ satisfying
	$$ \frac{G(h)}{T(h)}<c_0,
	$$
	we have, $\forall 0<h<h_0$
	$$ \|\chi_0(P(h))u\|_{\mathcal{H}}^2\leq C_0\frac{g(h)^2}{T(h)}\int_0^{T(h)}\|A(h)e^{-\frac{itP(h)}{h}}\chi_0(P(h))u\|_{\mathcal{H}}^2dt,
	$$
	where $\psi\in C_c^{\infty}((a,b))$.
\end{thm}

Let us prove Points (1) and (2) of Theorem \ref{t:main}, using Theorems \ref{t:resgamma} and \ref{abstract}. For $s\in\N$, $s\geq 1$, there holds:
$$
(-h^2\Delta_\gamma)^s-1=Q_{h,\gamma}(-h^2\Delta_\gamma-1)
$$
where 
$$
Q_{h,\gamma}=(-h^2\Delta_\gamma)^{s-1}+\ldots+1
$$
which is an elliptic operator, such that $Q_{h,\gamma}^{-1}$ is bounded from $L^2(M)$ to $L^2(M)$ (independently on $h$). Hence if
$$ (-h^2\Delta_\gamma)^s u_h-u_h=g_h,
$$
then
$$
-h^2\Delta_\gamma u_h-u_h=Q_{h,\gamma}^{-1}g_h
$$
and, applying Theorem \ref{t:resgamma} and using the $L^2(M)$ boundedness of $Q_{h,\gamma}^{-1}$, we get
$$ \|u_h\|_{L^2(M)}\leq O(1)\|u_h\|_{L^2(\omega)}+O(h^{-(\gamma+1)})\|g_h\|_{L^2(M)}.
$$

Let us now prove a {\it spectrally localized} observability inequality thanks to a rescaling argument. We first assume $s>\frac{\gamma+1}{2}$. The previous estimate gives 
$$ \|u_h\|_{L^2(M)}\leq O(1)\|u_h\|_{L^2(\omega)}+\frac{G(h)}{h}\|g_h\|_{L^2(M)}
$$
with $G(h)=o(h^{1-2s})$. Applying Theorem \ref{abstract} to $g(h)=1, A(h)=\mathbf{1}_{\omega}$, $P(h)=(-h^2\Delta_\gamma)^s$, by denoting $u_h=\chi_0((-h^2\Delta_\gamma)^s )u_0$ where $\chi_0\in C_c^\infty((1/2,2))$, we have
$$ \|u_h\|_{L^2(M)}^2\leq \frac{C_0}{T(h)}\int_0^{T(h)}\|e^{-\frac{it(-h^2\Delta_\gamma)^s}{h}}u_h\|_{L^2(\omega)}^2dt
$$
for $T(h)=C_1G(h)$ with $C_1=\frac{2}{c_0}$.
By changing variables $t'=h^{2s-1}t$, we have
$$ \|u_h\|_{L^2(M)}^2\leq \frac{C_0}{C_1G(h)h^{2s-1}}\int_0^{h^{2s-1}C_1G(h)}\|e^{-it'(-\Delta_\gamma)^s}u_h\|_{L^2(\omega)}^2dt'.
$$
Fix $T>0$. Since $h^{2s-1}G(h)=o(1)$ as $h\rightarrow 0$, we can apply the inequality above $\sim \frac{T}{C_1h^{2s-1}G(h)}$ times, together with the conservation of the $L^2(M)$ norm along the flow $e^{-it'(-\Delta_\gamma)^s}$. This yields
\begin{equation} \label{e:locobs}
 \|u_h\|_{L^2(M)}^2\leq C\int_0^{T}\|e^{-it'(-\Delta_\gamma)^s}u_h\|_{L^2(\omega)}^2dt'.
\end{equation}

In the case of Point (2) where $\frac12(\gamma+1)=s$, doing the same argument with $G(h)=O(h^{1-2s})$, we obtain that \eqref{e:locobs} holds only for $T$ sufficiently large, that is, $T\geq T_{\inf}$.

It remains to show how to deduce Point (1) (or Point (2)) from the localized observability inequality \eqref{e:locobs}. This procedure is standard (see \cite[Section 4]{BZ12}), but we recall it here briefly for the sake of completeness.

\subsection{From the localized observability to the full observability}

In the next lemma, $H_{\gamma}^{-1}(M)$ denotes the dual of $H_{\gamma,0}^1(M)$ (defined in Section \ref{s:mainresultssection}).
\begin{lemma}\label{embedding}
	The embeddings  $H_{\gamma,0}^1(M)\hookrightarrow L^2(M)$ and	$L^2(M)\hookrightarrow H_{\gamma}^{-1}(M)$ are compact.
\end{lemma}
\begin{proof}
	By duality, we only need to prove that $H_{\gamma,0}^1(M)\hookrightarrow L^2(M)$ is compact. Since $H_{\gamma,0}^1(M)\hookrightarrow H_{\gamma}^1(M)$, it suffices to show that $H_{\gamma}^1(M)\hookrightarrow L^2(M)$ is compact. For $s_1\in\N, s_2\geq 0$, denote by 
	$H^{s_1,s_2}(M)$ be the Sobolev space with respect to the norm
	$$ \|f\|_{H^{s_1,s_2}(M)}^2:=\|f\|_{L^2(M)}^2+\|\partial_x^{s_1}f\|_{L^2(M)}^2+\||D_y|^{s_2}f\|_{L^2(M)}^2.
	$$
	Note that
	$H_{\gamma}^1(M)=[L^2(M),H_{\gamma}^2(M)]_{\frac{1}{2}}$ and $H^{0,\frac{1}{\gamma+1}}(M)=[L^2(M),H^{0,\frac{2}{\gamma+1}}(M)]_{\frac{1}{2}}$, here $[\mathcal{X},\mathcal{Y}]_{\theta}$ is the standard notation of interpolation spaces (see Chapter 4 of \cite{Tay}).
	By Lemma \ref{l:dy}, we know that
    $H_{\gamma}^2(M)\hookrightarrow H^{0,\frac{2}{\gamma+1}}(M)$. Interpolating with the trivial embedding\footnote{Here we use the complex interpolation theorem, see for example \cite{LiP64}. } $L^2(M)\hookrightarrow L^2(M)$, we obtain that $H_{\gamma}^{1}(M)\hookrightarrow H^{0,\frac{1}{\gamma+1}}(M)$ is continuous.
	Moreover, since  $\|\partial_xu\|_{L^2(M)}\leq \|u\|_{H_{\gamma,0}^1(M)}$, we deduce that $H_{\gamma}^1(M)\hookrightarrow H^{1,\frac{1}{\gamma+1}}(M)$ is continuous. Thus from the compactness of the embedding $H^{1,\frac{1}{\gamma+1}}(M)\hookrightarrow L^2(M)$, we deduce that $H_{\gamma,0}^1(M)\hookrightarrow L^2(M)$ is compact.
\end{proof}

\begin{proof}[Proof of Point (1) and (2) of Theorem \ref{t:main}] Let $\psi(\rho):=\chi_0((-\rho)^s)$, hence $u_h=\psi(h^2\Delta_\gamma)u_0$.
	From \eqref{e:locobs}, we deduce that for sufficiently small $h_0>0$, $0<h<h_0$ and any $T>T_{\mathrm{inf}}$ (for Point (1) we say that $T_{\mathrm{inf}}=0$), there holds
	\begin{align}\label{eq1} 
	 \|\psi(h^2\Delta_{\gamma} )u_0\|_{L^2(M)}^2\leq C_T\int_{0}^{T}\|\phi_1e^{-it(-\Delta_{\gamma})^s}\psi(h^2\Delta_{\gamma} )u_0\|_{L^2(M)}^2dt,
		\end{align} 
	where supp$(\phi_1)\subset\omega$.
Taking $h=2^{-j}$ and summing over the inequality above for $j\geq j_0=\lfloor \log_2\left(h_0^{-1}\right)\rfloor+1$, by decreasing $h_0$ if necessary,
	we will get
	\begin{equation}\label{ob:weak}
	\|u_0\|_{L^2(M)}^2\leq C_T\int_{0}^T\|e^{-it(-\Delta_{\gamma})^s}u_0\|_{L^2(\omega)}^2dt+C_T\|\psi_0(2^{-2j_0}\Delta_{\gamma})u_0\|_{L^2(M)}^2,
	\end{equation}
	for some $\psi_0\in C_c^{\infty}(\R)$. To see this, we first take $\psi_0\in C_c^{\infty}(\R)$, equaling to $1$ on $(-\frac{1}{2},0]$. By the almost orthogonality, we have
	$$ \|u_0\|_{L^2(M)}^2\leq \|\psi_0(2^{-2j_0}\Delta_{\gamma})u_0\|_{L^2(M)}^2+C\sum_{j=j_0}^{\infty}\|\psi(2^{-2j}\Delta_{\gamma})u_0\|_{L^2(M)}^2\leq C\|u_0\|_{L^2(M)}^2.
	$$
	Applying \eqref{eq1}, we have for each $j\geq j_0$,
	\begin{align*}
	 &\|\psi(2^{-2j}\Delta_{\gamma})u_0\|_{L^2(M)}^2\\ \leq &C_T\int_0^T\|\psi(2^{-2j}\Delta_{\gamma})(\phi_1e^{-it(-\Delta_{\gamma})^s} u_0)\|_{L^2(M)}^2dt+C_T\int_0^T\|[\psi(2^{-2j}\Delta_{\gamma}),\phi_1]e^{-it(-\Delta_{\gamma})^s}u_0\|_{L^2(M)}^2dt\\
	 \leq &C_T\int_0^T\|\psi(2^{-2j}\Delta_{\gamma})(\phi_1e^{-it(-\Delta_{\gamma})^s} u_0)\|_{L^2(M)}^2dt+C_T2^{-2j}\|u_0\|_{L^2(M)}^2,
		\end{align*}
		where for the last step, we used the symbolic calculus for the commutator $[\psi(2^{-2j}\Delta_{\gamma}),\phi_1]$ and the fact that $e^{-it(-\Delta_{\gamma})^s}$ is unitary on $L^2(M)$.
	Summing the above inequality over $j\geq j_0$, we obtain \eqref{ob:weak}, provided that $h_0>0$ is chosen smaller so that $C_Th_0^2=C_T2^{-2j_0}<\frac{1}{2}$. Note that the second term on the right side of \eqref{ob:weak} can be  controlled by $\|u_0\|_{H_{\gamma}^{-1}(M)}^2$.
	
		 To conclude, we follow the approach of Bardos-Lebeau-Rauch \cite{BLR}. For $T'>0$, defining the set
	$$ \mathcal{N}_{T'}:=\left\{u_0\in L^2(M): e^{-it(-\Delta_{\gamma})^s}u_0|_{[0,T']\times\omega}=0 \right\}
	$$
	Let $T'\in (T_{\mathrm{inf}},T)$.  \eqref{ob:weak} implies that any function $u_0\in \mathcal{N}_{T'}$ satisfies
	$$ \|u_0\|_{L^2(M)}\leq C_T\|u_0\|_{H_{\gamma}^{-1}(M)}.
	$$
	Thanks to Lemma \ref{embedding}, we deduce that $\mathrm{dim}(\mathcal{N}_{T'})<\infty$. Note that for any $T_1<T_2$, $\mathcal{N}_{T_2}\subset \mathcal{N}_{T_1}$. Consider the mapping  $S(\delta):=\delta^{-1}\left(e^{-i\delta(-\Delta_{\gamma})^s}-\mathrm{Id}\right): \mathcal{N}_{T'}\rightarrow \mathcal{N}_{T'-\delta}$. For fixed $T'\in (T_{\mathrm{inf}},T)$, when $\delta< T'-T_{\mathrm{inf}}$, $\dim\mathcal{N}_{T'-\delta}<\infty$. Since the dimension is an integer, up to a slight diminution of $T'$, there exists $\delta_0>0$, such that for all $0<\delta<\delta_0$, $\mathcal{N}_{T'-\delta}=\mathcal{N}_{T'}$. Therefore, $S(\delta)$ is a linear map on $\mathcal{N}_{T'}$. Let $\delta\rightarrow 0$, we obtain that $$-i(-\Delta_{\gamma})^s|_{\mathcal{N}_{T'}}: \mathcal{N}_{T'}\rightarrow \mathcal{N}_{T'} $$ is a well-defined linear operator. Take any eigenvalue $\lambda\in\C$ of this operator, and assume that $u_*\in\mathcal{N}_{T'}$ is a corresponding eigenfunction (if it exists). There holds
	$$ (-\Delta_{\gamma})^su_{*}=i\lambda u_{*}.
	$$ 
	This implies that $u_*$ is an eigenfunction of $(-\Delta_{\gamma})^s$ (thus $u_*$ is also an eigenfunction of $-\Delta_{\gamma}$). However, $u_*|_{\omega}\equiv 0$, hence $u_*\equiv 0$ by the unique continuation property for $\Delta_\gamma$ (see \cite{Ga93}). Therefore,  $\mathcal{N}_{T'}=\{0\}$.

	Now we choose $T_0=T'$ as above. By contradiction, assume that Point (1) or Point (2) of Theorem \eqref{t:main} is untrue. Then there exists a sequence $(u_{k,0})_{k\in\N}$, such that 
	$$ \|u_{k,0}\|_{L^2(M)}=1,\quad \lim_{k\rightarrow\infty}\int_{0}^{T_0}\|e^{-it(-\Delta_{\gamma})^s}u_{k,0}\|_{L^2(M)}^2dt=0.
	$$
	Up to extraction of a subsequence, we may assume that $u_{k,0}\rightharpoonup u_0$ in $L^2(M)$. Thus from Lemma \ref{embedding}, $u_{k,0}\rightarrow u_0$, strongly in $H_{\gamma}^{-1}(M)$.  Since $e^{-it(-\Delta_{\gamma})^s}u_{k,0}\rightarrow 0$ in $L^2([0,T_0]\times\omega)$, we know that $u_0\in\mathcal{N}_{T_0}=\{0\}$. Besides, letting $k\rightarrow\infty$ in \eqref{ob:weak}, we obtain $\|u_0\|_{H_{\gamma}^{-1}(M)}>0$. This is a contradiction, which concludes the proof of Points (1) and (2) of Theorem \ref{t:main}. 
\end{proof}

\section{Theorem \ref{t:main}: proofs of non-observability} \label{s:Point3}

In this section, we prove the second part of Point (2) of Theorem \ref{t:main}, namely that observability fails for small times in case $s=\frac{\gamma+1}{2}$, and Point (3) also follows from this analysis. The proof is totally different from the proofs of observability presented in Section \ref{s:proof12}. Let us assume $\frac12(\gamma+1)=s$.\\ We note that if $\gamma=1$, then necessarily $s=1$, and the result was proved in \cite{BuSun}. Therefore, in the sequel, we assume $\gamma> 1$: this will be useful for establishing precise asymptotics of eigenfunctions, see Proposition \ref{p:esteiggen}.

The non-observability part of Point (2) immediately follows from:
\begin{prop} \label{p:gbpoint3}
There exist $T_0>0$ and a sequence of solutions $(u_n)_{n\in\N}$ of \eqref{e:schrodfrac} with initial data $(u_n^0)_{n\in\N}$  such that $\|u_n^0\|_{L^2(M)}=1$ and
\begin{equation} \label{e:unnotobs}
\int_0^{T_0}\int_\omega |u_n(t,x,y)|^2dxdydt \underset{n\rightarrow +\infty}{\longrightarrow} 0.
\end{equation}
\end{prop}
The goal of this section is to prove Proposition \ref{p:gbpoint3}. In all the sequel, using the invariance by $y$-translation, we assume without lost of generality that $\mathbb{T}_y\setminus I$ contains a neighborhood of $0$.

Here is a sketch of the proof, which borrows ideas from \cite[Section 9]{BuSun}:
\begin{itemize}
\item We can reduce the analysis to the construction of solutions of $i\partial_tu-(-\Delta_\gamma)^su=0$ in $\R\times\T$: then, using an appropriate cut-off, we transplant it into solutions of \eqref{e:schrodfrac} (thus in $(-1,1)\times\T$).
\item In $\R\times\T$ and for $\eta\in \R$, we consider as initial datum the normalized ground state $|\eta|^{\frac{1}{2(\gamma+1)}}\phi_\gamma(|\eta|^{\frac{1}{\gamma+1}}x)$ of the operator $D_x^2+|\eta|^2|x|^{2\gamma}$, mutiplied by $e^{iy\eta}$. The associated solution of $i\partial_tu-(-\Delta_\gamma)^su=0$ is obtained by mutiplication by a phase, and the intuition is that this solution has all its energy concentrated near $x=0$ when $\eta$ is large: it is analoguous to the ``degenerate regime'' of Section \ref{s:degenerate}. Taking linear combinations of these solutions for large $\eta$'s (this is the role of the multiplication by $\psi(h_nk)$ in \eqref{e:defvnpoint3}), we obtain a solution which travels along the $y$-axis at finite speed.
\end{itemize}

Let us now start the proof. The normalized ground state of the operator $P_{\gamma,w}=-\partial_x^2+|x|^{2\gamma}w^2$ on $\R_x$ is denoted by $p_\gamma(w,\cdot)$ and the associated eigenvalue is $\lambda_\gamma(w)$. We set $z=|w|^{\frac{1}{\gamma+1}}x$, and we are then left to study the operator $Q_\gamma=-\partial_z^2+|z|^{2\gamma}$ on $\R_z$. Recall that its normalized ground state is $\phi_\gamma$ which satisfies
\begin{equation*}
Q_\gamma\phi_\gamma=\mu_0\phi_\gamma
\end{equation*}
on $\R_z$. In particular, we have $\lambda_\gamma(w)=\mu_0|w|^{\frac{2}{\gamma+1}}$ and 
\begin{equation*}
p_\gamma(w,x)=|w|^{\frac{1}{2(\gamma+1)}}\phi_\gamma(|w|^{\frac{1}{\gamma+1}}x).
\end{equation*}

\begin{definition}\label{asymptotics} 
We write $f(x)=\widetilde{O}(g(x))$ as $x\rightarrow +\infty$ if for any $k\in\N$, 
$$  |f^{(k)}(x)|=O(|g^{(k)}(x)|),\quad x\rightarrow\infty. 
$$
\end{definition}

We need the following estimate, which is specific to the case $\gamma>1$:
\begin{prop} \label{p:esteiggen}We consider the ground state  
$$ -\phi_\gamma''+|z|^{2\gamma}\phi_\gamma=\mu_0\phi_\gamma, \quad \phi_\gamma(x)>0, \quad \phi_\gamma \text{ even}, \quad \|\phi_\gamma\|_{L^2(\R)}=1.
$$
Then, for some constant $c_\gamma\in\R$ we have the asymptotic behavior
\begin{equation}
\label{e:asymptbe} \phi_\gamma(x)\sim \frac{c_{\gamma }}{x^{\frac{\gamma}{2}}}e^{-\frac{x^{\gamma+1}}{\gamma+1}},\quad x\rightarrow\infty,
\end{equation}
and $\phi_\gamma=\widetilde{O}(x^{-\frac{\gamma}{2}}e^{-\frac{x^{\gamma+1}}{\gamma+1}})$.
\end{prop}
Proposition \ref{p:esteiggen} is proved in Section \ref{s:proofesteiggen} below, but let us first explain how to deduce Proposition \ref{p:gbpoint3} from these estimates.

\subsection{Estimate of the source term} 
We set $h_n=2^{-n}$ and we consider
\begin{equation}\label{e:defvnpoint3} v_n(t,x,y)=\sum_{k\in\Z}\psi(h_nk)e^{iyk-it\mu_0^s|k|^{\frac{2s}{\gamma+1}} }|k|^{\frac{1}{2(\gamma+1)}}\phi_\gamma(|k|^{\frac{1}{\gamma+1}}x),
\end{equation}
where $\psi\in C_c^{\infty}(\frac{1}{2}\leq \eta\leq 1)$,
and $\phi_\gamma$ is the first normalized eigenfunction of the operator $-\partial_x^2+|x|^{2\gamma}$ on $L^2(\R_x)$ with the eigenvalue $\mu_0>0$. Then $v_n$ satisfies
$$ i\partial_tv_n-(-\Delta_{\gamma})^sv_n=0
$$
on $\R_x\times \mathbb{T}_y$.

We consider a cut-off $\chi\in C_c^{\infty}(\R_x)$  with $\chi=1$ for $|x|\leq 1/4$ and $\chi(x)=0$ for $|x|\geq 1/2$. Let $u_n=\chi v_n$, then
$$ i\partial_tu_n-(-\Delta_{\gamma})^su_n=-[(-\Delta_{\gamma})^s,\chi]v_n.
$$
Our first goal is to show that 
\begin{equation} \label{e:smallbracket}
f_n:=[(-\Delta_{\gamma})^s,\chi]v_n\underset{ n\rightarrow+\infty}{\longrightarrow} 0, 
\end{equation}
in $L^2_{t,x,y}$ as $n\rightarrow+\infty$, uniformly in $t\in[0,T_0]$.

We write
\begin{equation} \label{e:commuts} [(-\Delta_{\gamma})^{s},\chi ]=\sum_{j=0}^{s-1} (-\Delta_{\gamma})^{j}[-\Delta_{\gamma},\chi](-\Delta_{\gamma})^{s-j-1}
\end{equation}
and we note that
$$ [-\Delta_{\gamma},\chi]=-2\chi'\partial_x-\chi''.
$$
Let us fix $0\leq j\leq s-1$ and focus on the term indexed by $j$ in \eqref{e:commuts}. We know that
\begin{equation}\label{ecommu2} 
[-\Delta_{\gamma},\chi](-\Delta_{\gamma})^{s-j-1}v_n(t,x,y)=\sum_{k\in\Z} \left(-2|k|^{\frac{1}{\gamma+1}}\phi_\gamma'(|k|^{\frac{1}{\gamma+1}}x)\chi'(x)-\phi_\gamma(|k|^{\frac{1}{\gamma+1}}x)\chi''(x)\right)\theta_n(t,y,k), 
\end{equation}
with
$$\theta_n(t,n,k)=\psi(h_n k)e^{ik y-it\mu_0^s|k|^{\frac{2s}{\gamma+1}} }(\mu_0|k|^{\frac{2}{\gamma+1}})^{s-j-1}|k|^{\frac{1}{2(\gamma+1)}}.$$
Now we have to take $j$ times $(-\Delta_\gamma)$ on the left of the above expression. For that, we determine the size of the new factors brought by any new $\partial_x$ or $|x|^\gamma\partial_y$ derivative. Indeed, we see that $(-\Delta_{\gamma})^j[-\Delta_{\gamma},\chi](-\Delta_{\gamma})^{s-j-1}v_n$ is a sum of terms of the form
\begin{equation*}
I^{n,j}_{j_1,j_2,j_3,j_4}(t,x,y):=\sum_{k\in\Z} |k|^{j_1}\phi_\gamma^{(j_2)}(|k|^{\frac{1}{\gamma+1}}x)\chi^{(j_3)}(x)(|x|^\gamma\partial_y)^{j_4}\theta_n(t,y,k)
\end{equation*}
with $j_1,j_2,j_3,j_4\geq 0$ bounded above by a constant which only depends on $s$. We also notice that necessarily $j_3\geq 1$, so that, with the properties of $\chi$, $I^{n,j}_{j_1,j_2,j_3,j_4}(t,x,y)=0$ for $|x|< 1/4$. Therefore, we can assume $|x|\geq 1/4$. Because of the term $\psi(h_n k)$ in $\theta_n(t,y,k)$, the sum in the definition of $I^{n,j}_{j_1,j_2,j_3,j_4}$ can be taken only over $k\in (h_n^{-1}/2,h_n^{-1})$. Now, using the profile of $\phi_\gamma^{(j_2)}$ given by Proposition \ref{p:esteiggen}, we see that $I^{n,j}_{j_1,j_2,j_3,j_4}=o(1)$ as $n\rightarrow +\infty$. Therefore, \eqref{e:smallbracket} holds.

By Duhamel's formula, we then have, for fixed $T_0>0$,
$$\|u_n(t)-e^{-it(-\Delta_{\gamma})^s}(\chi v_n(0))\|_{L^2_{x,y}}\underset{n\rightarrow +\infty}{\longrightarrow} 0
$$
uniformly in $t\in [0,T_0]$. Therefore, now considering $u_n$ as a function on $(-1,1)_x\times\mathbb{T}_y$, we see that \eqref{e:unnotobs} holds if and only if
\begin{equation} \label{e:vnnotobs}
\int_0^{T_0}\int_{\R\times I} |v_n(t,x,y)|^2dxdydt \underset{n\rightarrow +\infty}{\longrightarrow} 0.
\end{equation}

\subsection{Proof of \eqref{e:vnnotobs}} \label{s:vnnotobs} Recall that $\T_y\setminus I$ is assumed to contain a neighborhood of $0$. We prove that for any $c>0$, there exists $T_0>0$ such that $\|v_n\mathbf{1}_{|y|\geq c}\|_{L^2((0,T_0)\times \R_x\times\mathbb{T}_y)}\underset{n\rightarrow +\infty}{\longrightarrow} 0$, which implies \eqref{e:vnnotobs}. We consider the phase
\begin{equation*}
\Phi_m(t,y,w)=wy-\lambda_\gamma(w)^st-2\pi mw.
\end{equation*}
By the Poisson formula,
\begin{equation*}
v_n(t,x,y)=\sum_{m\in\mathbb{Z}}  \widehat{K_{t,x,y}^{(n)}}(2\pi m)
\end{equation*}
where 
\begin{equation} \label{e:K1}
 \widehat{K_{t,x,y}^{(n)}}(2\pi m)=\int_{\R}\psi(h_nw)p_\gamma(w,x)e^{i\Phi_m(t,y,w)}dw
\end{equation}
The goal is to prove that for $|y|\geq c$, each $ \widehat{K_{t,x,y}^{(n)}}(2\pi m)$ is small; therefore $v_n$ is also small for $y$ outside a neighborhood of $0$.

We do the usual integration by part argument, writing
\begin{equation}\label{e:Phin1}
e^{i\Phi_m(w)}=\frac{1}{i\partial_w\Phi_m}\frac{\partial}{\partial w}e^{i\Phi_m}.
\end{equation}
Here, using $\lambda_\gamma(w)=\mu_0|w|^{\frac{2}{\gamma+1}}$ and $s=\frac{\gamma+1}{2}$, we find
\begin{equation*}
\partial_w\Phi_m(t,y,w)=y-2\pi m- t\mu_0^s,
\end{equation*}
(for $w>0$) and in particular $\partial_w^2\Phi_m= 0$ (for $w>0$).
Using \eqref{e:Phin1}, we integrate by parts three times in \eqref{e:K1}:
\begin{equation} \label{e:valueK}
 \widehat{K_{t,x,y}^{(n)}}(2\pi m)=\frac{1}{i}\int_{\R}\frac{\partial^3_w(\psi(h_nw)p_\gamma(w,x))}{|\partial_w\Phi_m(t,y,w)|^3}e^{i\Phi_m}dw.
\end{equation}
There is a $\partial_w\Phi_m$ at the denominator, for which we need an estimate. We assume without loss of generality that $I\subsetneq \T_y$ is an interval, which we denote by $(a,b)$, with $0<a<b<\pi$. Let us fix $T_0<a/\mu_0^s$. We see that  $|\partial_w\Phi_m(t,y,w)|$ is bounded away from $0$ uniformly in $y\in I$ and $0\leq t\leq T_0$. Moreover $|m|\lesssim|\partial_w\Phi_m(t,y,w)|$ when $m\rightarrow +\infty$, uniformly in $w\in\R$, $y\in I$ and $0\leq t\leq T_0$. We write $|\partial_w\Phi_m(w)|\gtrsim |m-c_0|$ for some $0<c_0<1$ which does not depend on $m$. 

Let us now analyze \eqref{e:valueK}. Since on the support of $\psi(h_nw)$, $|w|\sim h_n^{-1}$, the main contribution of $\partial_w^3(\psi(h_nw)p_\gamma(w,x))$ comes from the situation where every derivative falls on the factor $\phi_{\gamma}(|w|^{\frac{1}{\gamma+1}}x)$, thus bounded by 
$$ O(|w|^{\frac{1}{2(\gamma+1)}+\frac{3}{\gamma+1}-3})=O(h_n^{\frac{3\gamma}{\gamma+1}-\frac{1}{2(\gamma+1)}}),\quad |w|\sim h_n^{-1}.
$$
Therefore, we obtain that
\begin{equation*}
\sup_{(t,x,y)\in (0,T)\times \omega} | \widehat{K_{t,x,y}^{(n)}}(2\pi m)|\leq \frac{Ch_n^{\frac{3\gamma}{\gamma+1}-\frac{1}{2(\gamma+1)}-1}}{|m-c_0|^3}.
\end{equation*}
Hence, the sum over $m$ of the $|\widehat{K_{t,x,y}^{(n)}}(2\pi m)|$ is $O(h_n^{\frac{4\gamma-3}{2(\gamma+1)} })$\footnote{Writing similar relations as \eqref{e:Phin1}, but at higher order, and then integrating by part sufficiently many times in \eqref{e:valueK}, we can obtain better bounds $O(h_n^{N})$ for any $N\in\N$.}. It gives \eqref{e:vnnotobs}, since $\gamma\geq 1$.
\begin{remark} \label{r:lowertime}
Note that this proof provides the lower bound $T_{\inf}\geq a/\mu_0^s$.
\end{remark}

\subsection{End of the proof of Proposition \ref{p:gbpoint3}} We finally need to estimate the size of the initial data.
\begin{lemma} \label{l:initial1}
There exists $c>0$ such that $\|v_{n,0}\|_{L^2(M)}\geq c$ for any $n\in\N$.
\end{lemma}
\begin{proof}[Proof of Lemma \ref{l:initial1}]
By Plancherel (used for fixed $x\in\R$), we have
\begin{align}
\|v_{n,0}\|_{L^2}^2&=\sum_{k\in\Z}\int_{\R} |\psi(h_nk)|^2|k|^{\frac{1}{\gamma+1}}|\phi_\gamma(|k|^{\frac{1}{\gamma+1}}x)|^2 dx \nonumber \\
&=\sum_{k\in\Z} |\psi(h_nk)|^2 \nonumber \gtrsim 1 \nonumber,
\end{align}
hence the conclusion.
\end{proof}

Combining Lemma \ref{l:initial1} and \eqref{e:vnnotobs}, we get Proposition \ref{p:gbpoint3}, and the non-observability part of Point (2) of Theorem \ref{t:resgamma} follows. Point (3) then follows immediately from the abstract result \cite[Corollary 3.9]{Mi12}: if \eqref{e:schrodfrac} was observable for some $T>0$ and some $s<\frac{\gamma+1}{2}$, then it would be observable in any time for $s=\frac{\gamma+1}{2}$, which is not the case thanks to the non-observability part of Point (2). 

\begin{remark}
Note that it would be possible to obtain Point (3) by a similar construction as the one of Section \ref{s:vnnotobs}: if $s<\frac{\gamma+1}{2}$, the phase $\partial_w\Phi_m$ verifies an estimate of the form $\partial_w\Phi_m=y-2\pi m+O(h_n^{1-\frac{2s}{\gamma+1}}T_0)$, and, since $h_n^{1-\frac{2s}{\gamma+1}}T_0$ tends to $0$ in any case as $n\rightarrow +\infty$, an analysis similar to the above one shows that observability fails for any $T_0>0$.
\end{remark}
\begin{remark}
The proof of Proposition \ref{p:gbpoint3} is adapted from the vertical Gaussian-beam like construction of \cite{BuSun} and this strategy was inspired by \cite{RS20} for the controllability of the Kadomtsev-Petviashvili equation. Since $s$ is a natural number, our construction here simplifies the analysis of Section 9 in \cite{BuSun}, without appealing to the properties of first eigenfunctions of the semi-classical generalized harmonic oscillators $-\partial_x^2+n^2|x|^{2\gamma}$ with Dirichlet boundary conditions. When $s$ is fractional, we do not have the nice formulas \eqref{e:commuts} and \eqref{ecommu2}, due to the non-local feature, and the analysis will be considerably more involved. Nevertheless, we believe that it is possible to handle a more precise analysis as in Section 9 of \cite{BuSun} to prove Point (3) for general $s>0$, not necessarily in $\N$. 
\end{remark}
\begin{remark}
It might be possible to generalize Proposition \ref{p:gbpoint3} to a more general setting thanks to a normal form procedure. By normal form, we mean that a complicated sub-Laplacian can sometimes be  (micro)-locally conjugated (by a Fourier Integral Operator) to a simpler one, see  \cite[Theorem 5.2]{CHT18} for the example of 3D contact sub-Laplacians. Since in the above proof of Point (3) the constructed sequence of solutions stays localized around a single fixed point of the manifold, we could hope to disprove observability for equations involving sub-Laplacians which are microlocally conjugated to $-\Delta_\gamma$.
\end{remark}

\subsection{Proof of Proposition \ref{p:esteiggen}} \label{s:proofesteiggen}
Our proof is inspired by \cite[Appendix IV]{Simon}. Note that we are only interested in the region $x\gg 1$. 
Let $Y=\binom{\psi}{\psi'}$, and
$$ A=\left(\begin{matrix}
0 & 1\\
|x|^{2\gamma}-\mu_0 &0
\end{matrix}
\right),
$$
hence $Q_\gamma\psi=\mu_0\psi$ is equivalent to $Y'=AY$. 
We set 
$$ \phi_-(x)=x^{-\gamma/2}e^{-\frac{x^{\gamma+1}}{\gamma+1}},\quad \phi_+(x)=x^{-\gamma/2}e^{\frac{x^{\gamma+1}}{\gamma+1}}.
$$
We compute
$$ \phi_-'(x)=-(x^{\frac{\gamma}{2}}+\frac{\gamma}{2}x^{-\frac{\gamma}{2}-1})e^{-\frac{x^{\gamma+1}}{\gamma+1}},\quad \phi_+'(x)=(x^{\frac{\gamma}{2}}-\frac{\gamma}{2}x^{-\frac{\gamma}{2}-1})e^{\frac{x^{\gamma+1}}{\gamma+1}},
$$ 
$$ \phi_-''(x)=(x^{\frac{3\gamma}{2}}+\frac{\gamma}{2}(\frac{\gamma}{2}+1)x^{-\frac{\gamma}{2}-2})e^{-\frac{x^{\gamma+1}}{\gamma+1}},\quad \phi_+''(x)=(x^{\frac{3\gamma}{2}}+\frac{\gamma}{2}(\frac{\gamma}{2}+1)x^{-\frac{\gamma}{2}-2})e^{\frac{x^{\gamma+1}}{\gamma+1}}.
$$
These two functions can be viewed as approximate solutions, as $x\rightarrow +\infty$, to 
$$ L\psi:=-\psi''+(x^{2\gamma}-\mu_0)\psi=0
$$
and we will give an expression of $\phi_\gamma$ in terms of $\phi_-$ and $\phi_+$, which will imply \eqref{e:asymptbe}. Let
\begin{align*}
U=\left(\begin{matrix}
\phi_-  &\phi_+\\
\phi_-' &\phi_+'
\end{matrix}\right)
\end{align*}
and $a=\binom{a_-}{a_+}:=U^{-1}Y$, or equivalently,
$$ \psi(x)=a_-(x)\phi_-(x)+a_+(x)\phi_+(x),\quad \psi'(x)=a_-(x)\phi'_-(x)+a_+(x)\phi'_+(x).
$$
 We remark that the inverse of $U$ exists since $\det(U)=\phi_+'\phi_--\phi_-'\phi_+=2$ and is given by
\begin{align*}
U^{-1}=\frac{1}{\mathrm{det}(U)}\left(\begin{matrix}
\phi_+'  &-\phi_+\\
-\phi_-' &\phi_-
\end{matrix}\right).
\end{align*}
We set the ansatz
$  Y=Ua,
$
hence $L\psi=0$ is equivalent to
$$ a'=-Ra,
$$
where
\begin{align*}
R=U^{-1}(U'U^{-1}-A)U=U^{-1}
\left(\begin{matrix}
0 &0\\
\mu_0+\frac{\gamma}{2}(\frac{\gamma}{2}+1)x^{-2} &0
\end{matrix}\right)U
\end{align*}
i.e.,
\begin{align*}
R=\frac{\mu_0+\frac{\gamma}{2}(\frac{\gamma}{2}+1)x^{-2}}{x^\gamma}\left(\begin{matrix}
-1 &-e^{\frac{2x^{\gamma+1}}{\gamma+1}}\\
e^{-\frac{2x^{\gamma+1}}{\gamma+1}} &1
\end{matrix}\right).
\end{align*}
To solve $a'=-Ra$, we expand the Neumann series as
$$ a(x)=\sum_{n=0}^{\infty}a_n(x),\quad a_{n}=\binom{a_{n,-}(x)}{a_{n,+}(x)}.
$$
where
$$ a_{n+1}(x)=\int_x^{\infty}R(z)a_n(z)dz,
$$
provided that the series and the integration converge. In order to avoid the divergence at $x=+\infty$, we initially choose
$$ a_0(x)=\binom{a_{0,-}}{0},
$$
where we can set $a_{0,-}=1$ is a constant. It turns out that the Neumann series  $a=\sum_{n=0}^{\infty}a_n$ converges to a smooth function $a$. Hence $Y=Ua$ is the solution of $Y'=AY$ which tends to $0$ as $x\rightarrow +\infty$. 
\begin{lemma}
There holds
$$ a_-(x)-1=e^{\frac{2x^{\gamma+1}}{\gamma+1}}\widetilde{O}(\frac{1}{x^{\gamma-1}}e^{-\frac{2x^{\gamma+1}}{\gamma+1}}), \quad a_+(x)=\widetilde{O}(\frac{1}{x^{\gamma-1}}e^{-\frac{2x^{\gamma+1}}{\gamma+1}}).
$$	
\end{lemma}
\begin{proof}
It follows from a simple recurrence that there exist $C>0$ and some (large) $x_0>0$ such that for any $n\in\N$ and any $x\geq x_0$, we have
\begin{equation*}
|a_{n,-}(x)|\leq \frac{C\mu_0^n}{x^{n(\gamma-1)}}, \qquad |a_{n,+}(x)|\leq \frac{C\mu_0^n}{x^{n(\gamma-1)}}e^{-2\frac{x^{\gamma+1}}{\gamma+1}}
\end{equation*}
It follows that $a_-(x)-1=O(1/x^{\gamma-1})$ and $a_+(x)=O(e^{-\frac{2x^{\gamma+1}}{\gamma+1}}/x^{\gamma-1})$. Then, the estimates on the derivatives of $a_-$ and $a_+$ follow from a recurrence using the relation $a'=-Ra$. 
\end{proof}
Thus we have constructed an explicit solution
$$ \psi_{\infty}(x):=a_-(x)\phi_-(x)+a_+(x)\phi_+(x)
$$
with the asymptotic behavior
$$ \psi_{\infty}(x)\sim x^{-\frac{\gamma}{2}}e^{-\frac{x^{\gamma+1}}{\gamma+1}},\quad x\rightarrow+\infty
$$
and $\psi_\infty=\widetilde{O}(x^{-\frac{\gamma}{2}}e^{-\frac{x^{\gamma+1}}{\gamma+1}})$.

Note that the Wronskian of the equation $L\psi=0$ is constant (so we can choose it to be $1$), so we find another independent solution (with some $x_0\gg 1$ fixed)
$$ \psi_{-\infty}(x):=\psi_{\infty}(x)\int_{x_0}^x\frac{dz}{(\psi_{\infty}(z))^2}\sim x^{-\frac{\gamma}{2}}e^{\frac{x^{\gamma+1}}{\gamma+1}}.
$$
Now the fundamental solution $\phi_\gamma(x)$ should be a linear combination of $\psi_{\infty}, \psi_{-\infty}$, namely, there exist constants $a,b\in\R$ such that
$$ \phi_\gamma(x)=a\psi_{\infty}(x)+b\psi_{-\infty}(x)
$$
for all large $x>x_0$ (this identity is only valid for large $x>0$). Since $\phi_\gamma(x)\rightarrow 0$ as $x\rightarrow+\infty$, we must have $b=0$, which finishes the proof.

\appendix

\section{Proof of the well-posedness} \label{a:wellposed}
We intend to prove the well-posedness of \eqref{e:schrodfrac}, \eqref{e:heatgamma} and \eqref{e:damped}.

\subsection{Schr\"odinger equation}
The equation \eqref{e:schrodfrac} can be solved by spectral theory. Expanding the initial datum $u_0(x,y)$ as 
\begin{equation} \label{e:splitu0}
u_0(x,y)=\sum_{j\in\N} a_{j}\varphi_j(x,y), \qquad \text{with } -\Delta_\gamma \varphi_j=\lambda_j^2\varphi_j,
\end{equation}
the solution of \eqref{e:schrodfrac} is given by
\begin{equation*}
(e^{-it(-\Delta_\gamma)^s}u_0)(t,x,y)=\sum_{j\in\N} a_{j}e^{-it\lambda_{j}^{2s}}\varphi_{j}(x,y), 
\end{equation*}
which belongs to $L^2(M)$ for any $t\in\R$.

Let us now prove \eqref{e:bdrycond}. For each $N$, we set
$$ u_N=\sum_{j\leq N}(u,\varphi_j)\varphi_j.
$$
Then, $(-\Delta_{\gamma})^ku_N|_{\partial M}=0$ for all $k\geq 0$. When $k\leq s$, 
$$ (-\Delta_{\gamma})^ku_N=\sum_{j\leq N}\lambda_j^{2k}(u,\varphi_j)\varphi_j 
$$
converges uniformly in $H_{\gamma}^{2(s-k)}(M)$ to $u$. When $k<s-\frac14$, since this is equivalent to $2(s-k)>\frac{1}{2}$, $(-\Delta_{\gamma})^ku_N|_{\partial M}$ converges in $L^2(\partial M)$ by trace theorem\footnote{Though $H_{\gamma}^s$ is not the usual Sobolev space, the usual trace theorem applies since near the boundary, $-\Delta_{\gamma}$ is uniformly elliptic.}. In particular, we have $(-\Delta_{\gamma})^ku|_{\partial M}=0$.
Note that when $s=\frac{k_0}{2}\in \frac{1}{2}\N$, $0\leq k<s-\frac{1}{4}$ is equivalent to $0\leq k\leq \big\lfloor \frac{k_0-1}{2}\big\rfloor$.

\subsection{Heat equation} To prove the well-posedness in $L^2(M)$, we will apply the Hille-Yosida theorem with generator $\widetilde{\mathcal{A}}=-(-\Delta_\gamma)^s$. The domain $D(\widetilde{A})$ is given by \eqref{e:domainpowers}, and it is dense in $L^2(M)$. For $u_0\in D(\widetilde{\mathcal{A}})$, written as in \eqref{e:splitu0}, there holds
\begin{equation*}
\Re(\langle \widetilde{\mathcal{A}}u_0,u_0\rangle_{L^2(M)})=-\sum_{j\in\N} |a_j|^2\lambda_{j}^{2s}\|\varphi_{j}\|_{L^2(M)}^2\leq 0,
\end{equation*}
hence $\widetilde{\mathcal{A}}$ is dissipative. Let us show that it is maximally dissipative, i.e., $\Id-\mu\widetilde{\mathcal{A}}$ is surjective for any $\mu>0$. Let $u_0$ as in \eqref{e:splitu0} and $\mu>0$. We consider
$$
u=\sum_{j\in\N} \frac{a_{j}}{1+\mu \lambda_{j}^{2s}}\varphi_{j}.
$$
Then $u\in L^2(M)$ and $(\Id-\mu\widetilde{\mathcal{A}})u=u_0$. Therefore, by the Hille-Yosida theorem, $\widetilde{\mathcal{A}}$ generates a strongly continuous semigroup of contraction, and in particular \eqref{e:heatgamma} is well-posed.

\subsection{Damped wave equation}
Consider the damped wave equation
$$  \partial_t^2u-\Delta_{\gamma}u+b\partial_tu=0
$$
where $b\in L^{\infty}(M)$ and $b\geq 0$. For its well-posedness in the energy space $\mathcal{H}=H_{\gamma,0}^1(M)\times L^2(M)$, 
we will apply the Hille-Yosida theorem to prove the existence and uniqueness of the semi-group $e^{t\mathcal{A}}$ with generator
$$ \mathcal{A}=
\left(\begin{matrix}
0  &1\\
\Delta_{\gamma} &-b
\end{matrix}
\right).
$$
We need to check the condition that $\mathcal{A}$ is maximally dissipative, which we formulate this time under the form
\begin{itemize}
	\item[(a)] $(0,\infty)\subset\rho(\mathcal{A})$;
	\item[(b)] $\|(\mu\Id-\mathcal{A})^{-1}\|_{\mathcal{L}(\mathcal{H})}\leq \mu^{-1}$, for any $\mu>0$.
\end{itemize}
Indeed, (a) is proved in the beginning of the proof of Corollary \ref{resolvent:dampedwave}. We only need to check (b). Let $U=(u,v)^t$ and $F=(u,v)^t$ such that $(\mu-\mathcal{A})U=F$.  Equipped with the inner product
$$ \big((u_1,v_1),(u_2,v_2)\big)_{\mathcal{H}}:=(\nabla_{\gamma}u_1,\nabla_{\gamma}u_2)_{L^2(M)}+(v_1,v_2)_{L^2(M)},
$$
we verify directly that
$$ \Re\big(\mathcal{A}U,U\big)_{\mathcal{H}}=-(bv,v)_{L^2(M)}\leq 0. 
$$
Therefore,
\begin{align*}
\mu\|U\|_{\mathcal{H}}^2\leq \mu(U,U)_{\mathcal{H}}-\Re(\mathcal{A}U,U)_{\mathcal{H}}=\Re((\mu\Id-\mathcal{A})U,U )_{\mathcal{H}}\leq \|U\|_{\mathcal{H}}\|(\mu\Id-\mathcal{A})U\|_{\mathcal{H}}.
\end{align*}
This means that $\mu \|(\mu\Id-\mathcal{A})^{-1}F\|_{\mathcal{H}}\leq \|F\|_{\mathcal{H}}$. Therefore, (b) is verified. The proof of well-posedness for the damped wave equation is then complete.

\section{Proof of Corollary \ref{c:damped}} \label{s:damped}

Recall that $\gamma\geq 1$ is fixed.
Given $b\in L^\infty(M)$, $b\geq0$, consider the damped wave equation
$$ \partial_t^2u-\Delta_{\gamma}u+b\partial_tu=0
$$
which can be written as $\partial_tU=\mathcal{A}U$ with $U=(u,\partial_tu)^t$ and
$$ \mathcal{A}=
\left(\begin{matrix}
0    &1\\
\Delta_{\gamma} &-b
\end{matrix}\right).
$$
Let $\mathcal{H}:=H_{0,\gamma}^1(M)\times L^2(M)$ and $H_{\gamma}^{-1}$ be the dual of $H_{0,\gamma}^1(M)$.
When $b=\mathbf{1}_{\omega}$, we have a stronger version of Theorem \ref{t:resgamma}:
\begin{prop}\label{t:resgamma'}
	There exist $C,h_0>0$, such that for all $0<h\leq h_0$, and any solution $v$ of
	$$ (h^2\Delta_{\gamma}+1)v=g_1+g_2,
	$$
	with $g_1\in L^2(M), g_2\in H_\gamma^{-1}$, we have 
	\begin{align*}
	\|h\nabla_{\gamma}v\|_{L^2(M)}+	\|v\|_{L^2(M)}\leq C\|v\mathbf{1}_{\omega}\|_{L^2(M)}+\frac{C}{h^{\gamma+1}}\|g_1\|_{L^2(M)}+\frac{C}{h^{\gamma+2}}\|g_2\|_{H_{\gamma}^{-1}(M)}.
	\end{align*}
\end{prop}

\begin{proof}
	Let $P_h=-h^2\Delta_{\gamma}-1+ih^{\gamma+1}$. We first show that $P_h$ is invertible.
	Note that for $v\in D(\Delta_{\gamma})$, we have
	$$ (P_hv,v)_{L^2(M)}=\|h\nabla_{\gamma}v\|_{L^2(M)}^2-\|v\|_{L^2(M)}^2+ih^{\gamma+1}(bv,v)_{L^2(M)}.
	$$
	Taking the imaginary part of the identity above, we have (using $b^2=b$)
	\begin{align}\label{imaginary} 
	\|bv\|_{L^2(M)}^2\leq h^{-(\gamma+1)}|\Im(P_hv,v)_{L^2(M)}|.
	\end{align}
	Taking the real part of the identity and inserting Theorem \ref{t:resgamma}, we have
	\begin{align*}
	\|h\nabla_{\gamma}v\|_{L^2(M)}^2+\|v\|_{L^2(M)}^2&\leq 2\|v\|_{L^2(M)}^2+|\Re(P_hv,v)_{L^2(M)}|\\
	&\leq C\|bv\|_{L^2(M)}^2+Ch^{-2(\gamma+1)}\|P_hv\|_{L^2(M)}^2+\|P_hv\|_{L^2(M)}\|v\|_{L^2(M)}.
	\end{align*}
	Applying Young's inequality and \eqref{imaginary}, we have
	$$ \|h\nabla_{\gamma}v\|_{L^2(M)}^2+\|v\|_{L^2(M)}^2\leq Ch^{-2(\gamma+1)}\|P_hv\|_{L^2(M)}^2.
	$$
	This implies that $P_h$ is invertible and 
	$$ P_h^{-1}=O(h^{-(\gamma+1)}):L^2(M)\rightarrow L^2(M),\qquad P_h^{-1}=O(h^{-(\gamma+2)}):L^2(M)\rightarrow H_{\gamma,0}^1(M).
	$$
	Now if $(h^2\Delta_{\gamma}+1)v=g_1+g_2$,  for any $w\in L^2(M)$, let $z=P_h^{-1}w$, and we have
	\begin{align*}
	(v,w)_{L^2(M)}=&(v,P_hz)_{L^2(M)}=(P_hv,z)_{L^2(M)}=(ih^{\gamma+1}b-g_1-g_2,z)_{L^2(M)}\\
	\leq &\|ih^{\gamma+1}b-g_1\|_{L^2(M)}\|z\|_{L^2(M)}+\|g_2\|_{H_{\gamma}^{-1}}\|z\|_{H_{\gamma,0}^1}\\
	\leq &Ch^{-(\gamma+1)}\|ih^{\gamma+1}b-g_1\|_{L^2(M)}\|w\|_{L^2(M)}+Ch^{-(\gamma+2)}\|g_2\|_{H_{\gamma}^{-1}}\|w\|_{L^2(M)}.
	\end{align*}
	Since $w\in L^2(M)$ is arbitrary, by duality, we complete the proof of Proposition \ref{t:resgamma'} 
\end{proof}

Consequently, the following resolvent estimate for the damped wave equation holds:
\begin{cor}\label{resolvent:dampedwave} 
	We have $i\R\subset\rho(\mathcal{A})$ and there exists $\lambda_0\geq 1$, such that for every $\lambda\in\R$, $|\lambda|\geq \lambda_0$, 
	\begin{align}\label{resolventbound}  \|(i\lambda{\rm Id}-\mathcal{A})^{-1}\|_{\mathcal{L}(\mathcal{H})}\leq C|\lambda|^{2\gamma}.
	\end{align}
\end{cor}

\begin{proof}[Proof of Corollary \ref{resolvent:dampedwave} from Theorem \ref{t:resgamma'}]
	
	We show that $i\R\subset\rho(\mathcal{A})$. This consists of two steps. First, we prove that $\mu\in \rho(\mathcal{A})$ for all $\mu>0$.
	Let $U=(u,v)^t$ and $F=(f,g)^t$, then $$(\mu{\rm Id}-\mathcal{A})U=F$$ is equivalent to
	\begin{align}\label{resolution} 
	\begin{cases}
	&\mu u-v=f\\
	&-\Delta_{\gamma}u+\mu v+bv=g,
	\end{cases}
	\end{align}
	hence $u$ satisfies the equation
	\begin{equation} \label{e:equ} -\Delta_{\gamma}u+(\mu b+\mu^2)u=g+(b+\mu) f.
	\end{equation}
	Consider the bilinear form on $H_{0,\gamma}^1$:
	$$ B_{\mu}[u,v]:=\Re(-\Delta_{\gamma}u+(\mu b+\mu^2)u,v )_{L^2(M)}=\Re\left((\nabla_{\gamma}u,\nabla_\gamma v)_{L^2(M)}+\mu^2(u,v)_{L^2(M)}+\mu(bu,v)_{L^2(M)}\right)
	$$
	which is coercive for all $\mu>0$. By Lax-Milgram, given $(f,g)\in\mathcal{H}$, \eqref{e:equ} posseses a unique solution $u\in H_{0,\gamma}^1$, and setting $v=\mu u-f$, we obtain a solution $(u,v)\in\mathcal{H}$ of \eqref{resolution}. Hence $\mu\in\rho(\mathcal{A})$. Moreover, we claim that $(\mathrm{Id}-\mathcal{A})^{-1}$ is compact. Indeed, from the equation of $u$, we deduce that $u\in H_{\gamma}^2(M)$. Since $v=\mu u-f$, we then deduce that $v\in H_{\gamma,0}^1(M)$. Now the compactness of $(\mathrm{Id}-\mathcal{A})^{-1}$ comes from the fact that the embedding $H_{\gamma}^{k+1}(M)\hookrightarrow H_{\gamma}^{k}(M)$ is compact (which we only need for $k=0,1$). 
	
	Now for any $z\in\C$, we write 
	$$ z-\mathcal{A}=(\mathrm{Id}+(1-z)(\mathcal{A}-\mathrm{Id})^{-1})(\mathrm{Id}-\mathcal{A}),
	$$
	since $\mathrm{Id}+(1-z)(\mathcal{A}-\mathrm{Id})^{-1}$ is Fredholm with index 0, we deduce that $z-\mathcal{A}$ is invertible (i.e. $z\in\rho(\mathcal{A})$) if and only if it is injective. To prove that $i\lambda-\mathcal{A}$ is injective for all $\lambda\in\R$, it suffices to show that any solution $u$ of
	$$ -\Delta_{\gamma}u-\lambda^2u+i\lambda bu=0
	$$ 
	is zero. Multiplying by $\overline{u}$, doing the integration by part and taking the imaginary part, we have
	$$ (bu,u)_{L^2}=0.
	$$ 
	Since $b=\mathbf{1}_{\omega}$, we have $bu=0$ a.e., hence we deduce that $u$ is an eigenfunction of $-\Delta_{\gamma}$ which vanishes on $\omega$. By the unique continuation property of $-\Delta_{\gamma}$ (see \cite{Ga93}), we deduce that $u\equiv 0$. This proves that $i\R\subset\rho(\mathcal{A})$.

	It remains to prove \eqref{resolventbound} for large $\lambda$. Without loss of generality, we assume that $\lambda\geq 1$. Let $U=(u,v)^t\in\mathcal{H}$ and $F=(f,g)^t\in\mathcal{H}$ such that  $(i\lambda-\mathcal{A})U=F$. Equivalently, with $h=\lambda^{-1}$,
	\begin{align*}
	\begin{cases}
	& u=-ih(v+f),\\
	& (h^2\Delta_{\gamma}+1)v=ihbv-ihg-h^2\Delta_{\gamma}f.
	\end{cases}
	\end{align*}
	Applying Theorem \ref{t:resgamma'} to $v$ and $g_1=ihg+ihbv, g_2=h^2\Delta_{\gamma}f$, we have
	\begin{align}\label{eq:resolution1}
	\|v\|_{L^2}\leq &C\|b^{\frac{1}{2}}v\|_{L^2}+Ch^{-(\gamma+1)}\|ihbv-ihg\|_{L^2}+Ch^{-(\gamma+2)}\|h^2\Delta_{\gamma}f\|_{H_{\gamma}^{-1}} \notag \\
	\leq &Ch^{-\gamma}\|b^{\frac{1}{2}}v\|_{L^2}+Ch^{-\gamma}\|g\|_{L^2}+Ch^{-\gamma}\|f\|_{H_{\gamma}^1}.
	\end{align}
	We need to estimate $\|b^{\frac{1}{2}}v\|_{L^2}$. Multiplying the equation $(h^2\Delta_{\gamma}+1)v=ihbv-ihg-h^2\Delta_{\gamma}f$ by $\overline{v}$, integrating it and taking the imaginary part, we have
	\begin{align*}
	(bv,v)_{L^2}\leq &|(g,v)_{L^2}|+h^{-1}|(h^2\Delta_{\gamma}f,v)_{L^2}|
	\leq \|g\|_{L^2}\|v\|_{L^2}+h\|\Delta_{\gamma}f\|_{H_{\gamma}^{-1}}\|v\|_{H_{\gamma}^1}\\
	\leq &\|g\|_{L^2}\|v\|_{L^2}+h\|f\|_{H_{\gamma}^1}\|ih^{-1}u-f\|_{H_{\gamma}^1}
	\leq \|g\|_{L^2}\|v\|_{L^2}+h\|f\|_{H_{\gamma}^1}^2+\|f\|_{H_{\gamma}^1}\|u\|_{H_{\gamma}^1}.
	\end{align*} 
	Plugging into \eqref{eq:resolution1} and using the fact that $\|b^{\frac{1}{2}}v\|_{L^2}^2=(bv,v)_{L^2}$ since $b\gtrsim \mathbf{1}_{\omega}$, we obtain that
	\begin{align}\label{eq2:resolution}
	\|v\|_{L^2}\leq & Ch^{-\gamma}\|g\|_{L^2}^{1/2}\|v\|_{L^2}^{1/2}+Ch^{-\gamma}\|f\|_{H_{\gamma}^1}^{1/2}\|u\|_{H_{\gamma}^1}^{1/2}+Ch^{-\gamma}\|g\|_{L^2}+Ch^{-\gamma}\|f\|_{H_{\gamma}^1}.
	\end{align}
	It remains to estimate $\|u\|_{H_{\gamma}^1}$. From the equation $u=-ihv-ihf$, we have
	$$ \|u\|_{H_{\gamma}^1}\leq h\|v\|_{H_{\gamma}^1}+h\|f\|_{H_{\gamma}^1}.
	$$
	Next, multiplying the equation $(h^2\Delta_{\gamma}+1)v=ihbv-ihg-h^2\Delta_{\gamma}f$ by $\overline{v}$, integrating it and taking the real part, we have
	\begin{align*}
	\|h\nabla_{\gamma}v\|_{L^2}^2\leq &\|v\|_{L^2}^2 +h|(g,v)_{L^2}|+ |(h^2\Delta_{\gamma}f,v)_{L^2}|\\
	\leq & \|v\|_{L^2}^2+Ch\|g\|_{L^2}^2+\frac{1}{2}h\|v\|_{L^2}^2+Ch^2\|\Delta_{\gamma}f\|_{H_{\gamma}^{-1}}^2+\frac{1}{2}h^2\|v\|_{H_{\gamma}^1}^2,
	\end{align*}
	hence
	$ \|hv\|_{H_{\gamma}^1}\leq Ch^{1/2}\|g\|_{L^2}+Ch\|f\|_{H_{\gamma}^1}+\|v\|_{L^2},
	$
	and
	$ \|u\|_{H_{\gamma}^1}\leq \|v\|_{L^2}+Ch\|f\|_{H_{\gamma}^1}+Ch^{1/2}\|g\|_{L^2}.
	$
	Plugging into \eqref{eq2:resolution} and using Young's inequality, we have
	$$ \|u\|_{H_{\gamma}^1}+\|v\|_{L^2}\leq Ch^{-2\gamma}\|g\|_{L^2}+Ch^{-2\gamma}\|f\|_{H_{\gamma}^1}.
	$$
	This completes the proof of Corollary \ref{resolvent:dampedwave}.
\end{proof}

Now, using \cite[Theorem 2.4]{BT10}, we obtain Corollary \ref{c:damped}.

\end{document}